\documentclass[a4paper, 12pt]{article}

\usepackage{amsmath} 
\usepackage{theorem}
\usepackage{amssymb}
\usepackage{amsfonts}
\usepackage{latexsym}
\usepackage{amscd}
\usepackage{diagramb}
\usepackage{yfonts}

\usepackage{tikz}
\usetikzlibrary{backgrounds}
\usetikzlibrary{arrows}
\usetikzlibrary{shapes,shapes.geometric,shapes.misc}
\pgfdeclarelayer{edgelayer}
\pgfdeclarelayer{nodelayer}
\pgfsetlayers{background,edgelayer,nodelayer,main}

\tikzset{baseline=(current bounding box.center)}
\tikzset{font=\scriptsize}
\tikzstyle{none}=[inner sep=0mm]
\tikzstyle{1-cell}=[draw=black, line width=1]
\tikzstyle{over-braiding}=[double, draw=white,line width=1.5,double=black,double distance=1]
\tikzstyle{2-cell}=[circle,draw=black, line width=1, fill=white, inner sep=0.5mm]

\input GrCalc4.sty
\usepackage{graphicx}
\usepackage{enumerate}
\usepackage{cancel}
\usepackage{extarrows}
\usepackage{upgreek}

\usepackage{pxfonts}
\usepackage{quiver} 
\usepackage[square,numbers,sort]{natbib}
\usepackage{stmaryrd}

\input xy
\xyoption{all}

\DeclareMathAlphabet{\mathpzc}{OT1}{pzc}{m}{it}

\usepackage{xspace,colortbl}
\usepackage{color}

\definecolor{verde}{rgb}{0.,0.7,0.}
\definecolor{indigo}{rgb}{.18, .34, .78}
\definecolor{indigo1}{rgb}{.18, .24, .78}
\definecolor{indigo2}{rgb}{.18, .14, .78}
\definecolor{indigo3}{rgb}{.18, 0., .78}
\definecolor{rojo}{rgb}{1,0,0}
\definecolor{negro}{rgb}{0,0,0}
\definecolor{lila}{rgb}{.46, .16, .78}
\definecolor{lila1}{rgb}{.46, .16, .86}
\definecolor{lila2}{rgb}{.56, .16, .86}
	\definecolor{lila3}{rgb}{.63, .16, .78}
\definecolor{lila4}{rgb}{.7, .16, .78}
\definecolor{lila5}{rgb}{.78, .26, .78}
\definecolor{lila6}{rgb}{.6, 0., .78}

\theoremstyle{definition}
\theoremheaderfont{\normalfont\bfseries}

\newtheorem{thm}{Theorem}[section]
\newtheorem{lma}[thm]{Lemma}
\newtheorem{cor}[thm]{Corollary}

\theorembodyfont{\rmfamily}
\newtheorem{defn}[thm]{Definition}

\theorembodyfont{\rmfamily}

\newtheorem{rem}[thm]{Remark}
\newtheorem{prop}[thm]{Proposition}
\newtheorem{ex}[thm]{Example}
\newcommand{\qed}{\hfill\quad\fbox{\rule[0mm]{0,0cm}{0,0mm}}  \par\bigskip}

\newcommand{\x}{\mbox{-}}

\newcommand{\w}{\hspace{-0,06cm}}
\newcommand{\s}{\hspace{0,06cm}}

\newcommand{\Ob}{{\mathcal Ob}}

\newcommand{\bEM}{{\rm bEM}}
\newcommand{\EM}{{\rm EM}}

\newcommand{\Mnd}{{\rm Mnd}}
\newcommand{\Comnd}{{\rm Comnd}}
\newcommand{\Bimnd}{{\rm Bimnd}}
\newcommand{\Bilax}{{\rm Bilax}}

\newcommand{\Ps}{{\rm Ps}}

\newcommand{\Bicat}{{\rm Bicat}}

\newcommand{\Pseudo}{{\rm Pseudo}}

\newcommand{\TF}{{\rm TF}}

\newcommand{\comp}{\circ}
\newcommand{\iso}{\cong}
\newcommand{\ot}{\otimes}

\newcommand{\C}{{\mathcal C}}
\newcommand{\Tau}{{\mathcal T}}
\newcommand{\M}{{\mathcal M}}

\newcommand{\Pp}{{\mathcal P}}
\newcommand{\D}{{\mathcal D}}
\newcommand{\F}{{\mathcal F}}
\newcommand{\G}{{\mathcal G}}

\newcommand{\HH}{{\mathcal H}}
\newcommand{\N}{{\mathcal N}}
\newcommand{\A}{{\mathcal A}}
\newcommand{\B}{{\mathcal B}}

\newcommand{\E}{{\mathcal E}}

\newcommand{\T}{{\mathcal T}}

\newcommand{\YD}{{\mathcal YD}}

\newcommand{\Id}{\operatorname {Id}}
\newcommand{\id}{\operatorname {id}}

\newcommand{\Epsilon}{\varepsilon}

\def\Z{{\mathcal Z}}  
\def\K{{\mathcal K}}  

\newcommand{\Mod}{\operatorname{Mod}}
\newcommand{\Bimod}{\operatorname{Bimod}}
\newcommand{\Lax}{\operatorname{Lax}}
\newcommand{\Dist}{\operatorname{Dist}}

\newcommand{\cref}[1]{C.~\ref{c:#1}}

\newcommand{\exlabel}[1]{\label{ex:#1}}
\newcommand{\exref}[1]{Example~\ref{ex:#1}}
\newcommand{\lelabel}[1]{\label{le:#1}}
\newcommand{\leref}[1]{Lemma~\ref{le:#1}}
\newcommand{\eqlabel}[1]{\label{eq:#1}}
\newcommand{\equref}[1]{(\ref{eq:#1})}

\newcommand{\delabel}[1]{\label{de:#1}}
\newcommand{\deref}[1]{Definition~\ref{de:#1}}
\newcommand{\prlabel}[1]{\label{pr:#1}}
\newcommand{\prref}[1]{Proposition~\ref{pr:#1}}
\newcommand{\colabel}[1]{\label{co:#1}}

\newcommand{\rmlabel}[1]{\label{rm:#1}}
\newcommand{\rmref}[1]{Remark~\ref{rm:#1}}
\newcommand{\selabel}[1]{\label{se:#1}}
\newcommand{\seref}[1]{Section~\ref{se:#1}}
\newcommand{\sslabel}[1]{\label{ss:#1}}
\newcommand{\ssref}[1]{Subsection~\ref{ss:#1}}

\newdir{:=}{{}}
\newcommand{\fit}[3]{\ar@{:=}@/{#3}/[#1] |{\Downarrow #2} }

\definecolor{pastel-red}{RGB}{175,53,30}

\title{Categorical centers and Yetter--Drinfel`d-modules as 2-categorical (bi)lax structures} 
\date{}
\author{%
Bojana Femi\'c\thanks{Corresponding author} \vspace{6pt} \\
{\small Mathematical Institute of  \vspace{-2pt}}\\ 
{\small Serbian Academy of Sciences and Arts } \vspace{-2pt}\\
{\small Kneza Mihaila 36,} \vspace{-2pt}\\
{\small 11 000 Belgrade, Serbia}\\
{\small femicenelsur@gmail.com} 
\and
Sebastian Halbig\vspace{6pt} \\
{\small Research Group Algebraic Lie Theory}\vspace{-2pt}\\
{\small Philipps-Universit\"at Marburg} \vspace{-2pt}\\
{\small  Biegenstra\ss e 10, } \vspace{-2pt}\\
{\small  35037 Marburg, Germany}\\
{\small Sebastian.Halbig@uni-marburg.de}
}

\begin{document}

\maketitle

\begin{abstract}
The bicategorical point of view provides a natural setting for many concepts in the representation theory of monoidal categories.  
We show that centers of twisted bimodule categories correspond to categories of 2-dimensional natural transformations and modifications between the deloopings of the twisting functors. We also show that dualities lift to centers of twisted bimodule categories. 
Inspired by the notion of (pre)bimonoidal functors due to McCurdy and Street and by bilax functors of Aguiar and Mahajan,  
we study 2-dimensional functors which are simultaneously lax and colax with a compatibility condition. Our approach uses a 
sort of 2-categorical Yang-Baxter operators, but the idea could equally be carried out using a kind of 2-categorical braidings. 
We show how this concept, which we call bilax functors, generalize many known notions from the theory of Hopf algebras. 
We propose a 2-category of bilax functors whose 1-cells generalize the notions of Yetter-Drinfel`d modules in ordinary categories, 
and a type of bimonads and mixed distributive laws in 2-categories. We show that the 2-category of bilax functors from the trivial 
2-category is isomorphic to the 2-category of bimonads, and that there is a faithful 2-functor from the latter to the 2-category 
of mixed distributive laws of Power and Watanabe. 
\end{abstract}


{\small {\em 2020 MSC:} 18N10, 18D25, 18M15.

\section{Introduction}\selabel{intro}

The concept of centers of a monoids was categorified independently by Drinfel`d, Majid and Street in the 1990's.
Since then it has been extensively studied in Hopf algebra and category theory, see for example \cite{Kassel} for an overview.
One of its striking features comes from the fact that by passing from sets to categories one can replace the qualitative question: `Do two elements commute with another?' with a quantitative one: `How many suitably coherent (iso-)morphisms
exist between the tensor product of two objects and its opposite?'. 
Such (iso-)morphisms are called \emph{half-braidings}.
The center of a monoidal category consists of objects of the underlying category with fixed half-braidings together with morphisms of the base category which satisfy a certain compatibility relation.

The aim of the present paper is twofold. In the first part, we study the center construction from the bicategorical perspective.
Our main motivation comes from the observation that monoidal categories can be identified with bicategories with a single object.
This procedure, sometimes called delooping, establishes an equivalence between monoidal categories with monoidal functors and bicategories with a single object together with pseudofunctors.
Following this line of thinking, by simple means we reveal a surprising and beautiful fact that colax natural transformations between lax 
functors among bicategories with single objects are nothing but the objects of the twisted Drinfel`d center of the corresponding codomain monoidal category. Accordingly, modifications of such colax transformations correspond to the morphisms in the Drinfel`d center, so that one has an isomorphism of categories. In particular, we obtain:

\begin{thm} 
Let $\C$ be a monoidal category and write $Del(\C)$ for its delooping. 
There exists a monoidal equivalence of categories between the Drinfel`d center of $\C$ and the category of pseudonatural 
transformations and their modifications on the identitity functor of $Del(\C)$.
\end{thm}

In fact, we formulate a general (weak) center category $\Z^w(F,\M,G)$ for a $\D\x\C$-bimodule category $\M$ and two lax monoidal functors 
$F:\E\to\D$ and $G:\E\to\C$ from a third monoidal category $\E$. In the case that $\C=\D=\M$, we interpret $\Z^w(F,\D,G)$ from a 
bicategorical point of view. As this interpretation relies on delooping, in paricular on the fact that the monoidal product of $\C$ becomes the composition of 1-cells, this approach does not allow for giving a bicategorical interpretation of $\Z^w(F,\M,G)$ for general $F$ and $G$. 
This task will be treated elsewhere. (The weakness 
corresponds to dealing with non-invertible half-braidings, while with strongness we allude to invertible ones. Accordingly, we 
differentiate left and right weak centers. )

We prove that left weak center categories form a bicategory $\Z^w_l(\E,\K)$ which is isomorphic to the bicategory of suitable lax 
functors of bicategories. 
When $\K$ is an autonomous 2-category, meaning that all its 1-cells have left and right adjoints, and $\E$ is 
an autonomous monoidal category, then on pseudofunctors $Del(\E)\to\K$ the weak and strong center categories coincide (\prref{w=s}). 
Furthermore, under these conditions the corresponding bicategory $\Z^{w\x ps}_l(\E,\K)$ is autonomous (\prref{auton}). 
This result provides a natural interpretation of duality notions between centers of compatible bimodule categories in \cite{HZ}.
Moreover, our bicategorical interpretation of center categories $\Z^w_l(F,\M,G)$ which make up the bicategory $\Z^w_l(\E,\K)$ 
encompasses also the Shimizu's bicategory $\TF(\C,\D)$ of tensor functors from \cite{Shim} and the result thereof about duals. 

The above bicategory $\Z^w_l(\E,\K)$ is a particular case of the bicategory $\Lax_{clx}(\B,\B')$ of lax functors $\B\to\B'$ 
among bicategories, colax natural transformations and modifications. Taking the pseudo-pseudo version of the latter, 
we get the bicategory $\Ps_{ps}(\B,\B')$ whose hom-categories are $\Ps(\F,\G)$ for pseudofunctors $\F,\G:\B\to\B'$. 
When $\F=\G=\Id_\B$ one recovers the center category $\Z(\B)$ of the bicategory $\B$ introduced in 
\cite{Ehud}. Our above bicategory of center categories alludes to the possibility to consider ``twisted center categories of the bicategory $\B'$''. 

\smallskip

On the other hand, in the second part of the paper, 
we introduce and study 2-categorical functors which are simultaneously lax and colax with a compatibility relation involving a 
Yang--Baxter operator. 
We call them {\em bilax functors} and differentiate {\em bilax functors with compatible Yang-Baxter operator}. 
For monoidal categories such functors were studied under the name of pre-bimonoidal and bimonoidal functors in \cite{CS}, 
and (when the domain category is braided) bilax functors in \cite{Agui}. 
We show that our bilax functors generalize a variety of notions and possess certain preservation properties:  
bialgebras in braided monoidal categories, {\em bimonads} 
in 2-categories (with respect to Yang-Baxter operators, YBO's), and preserve bimonads (w.r.t. YBO's), 
bimonads in 2-categories with respect to distributive laws from \cite{F1}, {\em module comonads} and 
{\em comodule monads}, and {\em relative bimonad modules}. Moreover, the component functors of a bilax functor on hom-categories 
factor through the category of {\em Hopf bimodules} (w.r.t. YBO's). The 2-categorical notions in italic letters are introduced in this paper 
and they generalize to 2-categories the same named notions in braided monoidal categories. 

We record that instead of working with Yang-Baxter operators, one could equally use local braidings, following the footsteps of \cite{Agui}. 
In this case the generalization and preservation results somewhat differ from the ones that we obtained and that are listed above.

We establish a 2-category of bilax functors $\Bilax(\K,\K')$ by introducing bilax natural transformations and bilax modifications. 
Accordingly, $\Bilax_c(\K,\K')$ denotes the 2-category of bilax functors with compatible Yang-Baxter operator. 
Bilax natural transformations are both lax and colax natural transformations satisfying a compatibility condition. 
As such they 
generalize bimonad morphisms from \cite{F1} and Yetter-Drinfel`d modules from braided monoidal categories. 
In the classical case, the category of Yetter-Drinfel`d modules over a bialgebra $B$ is monoidally equivalent to the Drinfel`d center of 
the category ${}_B\M$ of modules over the same bialgebra. The half-braidings in the left Drinfel`d center can be seen as colax natural transformations. 
In the category ${}_B\M$ 
one can construct a lax natural transformation which together with the colax one makes a bilax natural transformation. 
(We explain this in more detail at the end of \ssref{b.nat-tr}.)  
This illustrates why in a general 2-category $\K$ 
bilax natural transformations (and bilax modifications) generalize the category of Yetter-Drinfel`d modules, but not the (left) center category. 
Finally, we show that there is a 2-category isomorphism $\Bilax_c(1,\K)\iso\Bimnd(\K)$ and a faithful 2-functor 
$\Bimnd(\K)\hookrightarrow\Dist(\K)$. Here $\Bimnd(\K)$ is the 2-category of bimonads from \cite{F2} and $\Dist(\K)$ is the 2-category 
of mixed distributive laws of \cite{PW}. 

\medskip

The paper is composed as follows. We first give an overview of bicategories, deloopings, module and center categories. In 
section 3 we give a higher categorical interpretation of center categories and study when the bicategory of center categories 
is autonomous. Bilax functors and their properties are studied in section 4, while in the last section a 2-category of bilax 
functors is introduced and its relations to the 2-categories $\Bimnd(\K)$ and $\Dist(\K)$ is shown.

\section{Preliminaries: Deloopings and weak twisted centers} \selabel{prelim}

We assume that the reader is familiar with the notion of a braided monoidal category and the corresponding 
notation of string diagrams (see {\em e.g.} \cite{JS, Kassel, Tak}), as well as with the definition of a bicategory, 
for which we recommend \cite{Ben,JY}. 

In this section we give a short summary of bicategories, delooping bicategories, module categories and weak twisted centers. 
For a more extensive discussion of module categories we refer the reader to \cite{EGNO}. \\

Briefly, a \emph{monoidal category} consists of a category $\C$ together with a suitably associative and unital multiplication $\otimes \colon \C \times \C \to \C$ implemented by a functor which is called the \emph{tensor product}.

A `many object' generalization of monoidal categories is provided by \emph{bicategories}. These can be thought of as higher dimensional categories with \emph{hom-categories} between every pair of objects instead of mere sets. The objects of these hom-categories are called 
\emph{1-cells} and the morphisms \emph{2-cells}.
Any bicategory $\K$ admits two ways to compose: \emph{horizontal composition} given by the composition functors 
\begin{equation*}
\circ_{Z,Y,X}\colon \K(Y,Z)\times \K(X,Y) \to \K(X, Z), \qquad\qquad \text{ for } X,Y,Z\in \Ob\K \text{ (objects of $\K$)},
\end{equation*}
and \emph{vertical composition} induced by the compositions inside the hom-categories. 
Instead of identity morphisms, every $X\in\Ob\K$ has a unit 1-cell $\id_{X}\in \K(X,X)$. 
In general, the horizontal composition of a bicategory is associative and unital only up to suitable natural isomorphisms.
Bicategories where these morphisms are identities are called \emph{2-categories}.
Since every bicategory is biequivalent to a  2-category, we will restrict ourselves without loss of generality 
to the setting of 2-categories. 

As hinted at before, there is an intimate relationship between monoidal categories and bicategories. It is provided by considering a monoidal category $\C$ as a bicategory $Del(\C)$ with one object (which we will usually denote by $*$) and $\C$ as its unique hom-category. Under this identification, the tensor product of $\C$ becomes the horizontal composition of $Del(\C)$, and the monoidal unit
becomes the identity 1-cell $id$ on the unique object of $Del(\C)$.
The resulting canoncial isomorphism of categories between the category of monoidal categories with certain structure preserving functors and one-object bicategories plus structure preserving 2-dimensional functors is called \emph{delooping}:
\begin{equation} \eqlabel{eq:delooping}
  Del \colon \left\{\begin{gathered}
       \text{monoidal categories with} \\ \text{lax/colax/strong monoidal functors}
     \end{gathered}\right\} \to
  \left\{
    \begin{gathered}
      \text{one-object bicategories with} \\
      \text{lax/colax/pseudofunctors}.
    \end{gathered}\right\}
\end{equation}


\begin{rem} \rmlabel{reversed}
Observe that we consider the horizontal composition in bicategories in the counter lexicographical order, whereas 
the tensor product in a monoidal category $\C$ is read from left to right, that is: $\C\times\C\ni(X,Y)\stackrel{\ot}{\mapsto} X\ot Y\in\C$. 
For objects $X,Y\in\C$ corresponding to 1-cells $x,y$ in $Del(\C)$ respectively, this implies that the tensor product $X\ot Y$ corresponds to the composition of 1-cells $y\circ x$. In order to avoid applying this mirror symmetry, we are going to consider in the isomorphism $Del$ 
in \equref{eq:delooping} that the reversed tensor product becomes the horizontal composition in bicategories (formally, this is precomposing 
$Del$ with the isomorphism functor defined on objects by sending a monoidal category $(\C, \ot)$ to it reversed category$(\C^{rev}, \ot^{rev})$ ). 
\end{rem}

Bicategories provide a natural interpretation of the representation theory of monoidal categories. 
All endomorphism categories of a bicategory are monoidal with horizontal composition as a tensor product.
Similarly, given two objects $A, B \in \Ob\K$ of a bicategory with endomorphism categories $\D:=\K(A,A)$ and $\C:=\K(B,B)$, 
horizontal composition endows $\M \coloneqq \K(A,B)$ with the structure of a $(\C, \D)$-\emph{bimodule category}.
That is, there are two functors $\rhd\colon \C \times \M \to \M$ and $\lhd \colon \M \times \D \to \D$ subject to analogous but weakened version of the axioms of bimodules over a monoid. 

Conversely: to any $(\C, \D)$-bimodule category $\M$ we can associate a two object bicategory $Del(\M)$, which we call the
\emph{delooping} of $\M$. It has two objects $0$ and $1$ and hom-categories $Del(\M)(0,0) = \D, Del(\M)(0,1)=M, Del(\M)(1,1) =\C$ and 
$Del(\M)(1,0) =1$, the trivial category. Horizontal composition is given by 
the tensor products of $\C$ and $\D$ and the left and right action of $\C$ and $\D$ on $\M$. The relation between (bi)module categories and bicategories 
was already observed by Benabo\'u, \cite[Section 2.3]{Ben}.

If the categories $\C$ and $\D$ coincide, one can define the center of a bimodule category.
The aim of the paper at hand will be the study of these centers and their interaction with the theory of bicategories in a slightly more general version.

\begin{defn} \delabel{relative center cat}
  Let $F \colon \E \to \D$ and $G \colon \E \to \C$ be lax monoidal functors and $\M$ a (strict) $(\C, \D)$-bimodule category over the (strict) monoidal categories $\C$ and $\D$.
  A \emph{left half-braiding} of an object $M\in \M$  relative to $F$ and $G$ is a natural transformation
  \begin{equation*}
   \sigma_{X} \colon M \lhd F (X) \to G (X) \rhd M, \qquad \qquad \text{for all } X \in \E,
 \end{equation*}
 such that for all $X,Y \in \mathcal{E}$ the following diagrams commute:
 \begin{equation} \eqlabel{center diag1} 
 \begin{tikzcd}
	{M \lhd F(Y) \lhd F(X)} && {G(Y) \rhd M \lhd F(X)} \\
	{M  \lhd F(Y \otimes X) } & {G(Y \otimes X) \rhd M } & {G(Y) \rhd G(X) \rhd M }
	\arrow["{\sigma_Y \lhd \id_{F(X)}}", from=1-1, to=1-3]
	\arrow["{\id_{G(Y)}\lhd \sigma_X}", from=1-3, to=2-3]
	\arrow["{G^2 \rhd M}", from=2-3, to=2-2]
	\arrow["{M \lhd F^2}"', from=1-1, to=2-1]
	\arrow["{\sigma_{Y\otimes X}}"', from=2-1, to=2-2]
\end{tikzcd}
\end{equation}
\begin{equation} \eqlabel{center diag2}  
\begin{tikzcd}
	{ M \lhd I \cong I \rhd M} && {G(I) \rhd M} \\
	& {M\lhd F(I)}
	\arrow["{G^0 \rhd M}", from=1-1, to=1-3]
	\arrow["{M \lhd F^0}"', from=1-1, to=2-2]
	\arrow["{\sigma_I}"', from=2-2, to=1-3]
\end{tikzcd}
\end{equation}
Similarly, a \emph{right half-braiding} on $M$ relative to $F$ and $G$ is a natural transformation
\begin{equation*}
  \tilde\sigma_{X} \colon G (X) \rhd M \to M \lhd F (X), \qquad \qquad \text{for all } X \in \E,
\end{equation*}
subject to analogous identities.
\end{defn}

The \emph{left weak center of $\M$ relative to $F$ and $G$} is the category $\Z_l^w(F, \M, G)$.
Its objects are pairs $(M,\sigma)$ consisting of an object $M \in\M$ together with a left half-braiding $\sigma$ on $M$ relative to $F$ and $G$.
A morphism between objects $(M, \sigma), (N, \tau) \in \Z_{l}^{w}(F, \M, G)$ is an arrow $f \in \M(M,N)$ such that
\begin{equation} \label{eq:morph-of-center}
  (\id_{G(X)}\rhd f)\sigma_{X} = \tau_{X}(f \lhd \id_{F(X)}), \qquad \qquad \text{for all } X \in \E.
\end{equation}

The full subcategory $\Z^{s}_{l}(F, \M, G)$ of $\Z_{l}^{w}(F, \M, G)$ whose objects have invertible half-braidings is called the \emph{(strong) left center of $\M$ relative to $F$ and $G$}. When the functors are clear from the context, we will call the latter two categories simply 
{\em left weak/strong twisted centers of $\M$}, respectively.  

We define the right weak and strong twisted center categories $\Z_{r}^{w}(F, \M, G)$ and $\Z^{s}_{r}(F, \M, G)$ in an analogous way. 

When $\C=\D=\M$ we set $Z_{l}^{w}(F, G):=\Z_{l}^{w}(F, \D, G)$ and $Z_{r}^{w}(F, G):=\Z_{r}^{w}(F, \D, G)$. 
For $\C$ a tensor category and $F,G$ tensor functors these present the (left and right version of) twisted center category $Z(F,G)$ 
studied in \cite[Section 3]{Shim}.

In case $F=G=\Id_\C$, we write $\Z^l_\C(\M)\coloneqq \Z^s_l(\Id,\M,\Id)$ and $\Z^r_\C(\M)\coloneqq\Z^s_r(\Id,\M,\Id)$. 
These recover the (left and right) center category from \cite{GNN}. If moreover $\M=\C$, the categories $\Z^l_\C(\C)$ and $\Z^r_\C(\C)$ 
recover the left and right Drinfel`d center categories of $\C$.

\begin{lma} \lelabel{left-iso-right}
  Suppose $G\colon \E \to \C$ and $F\colon \E \to \D$ are lax monoidal functors and $\M$ is a $(\C, \D)$-bimodule category. 
  Then there exists an isomorphism of categories
  \begin{equation} \label{eq:}
    \Xi \colon \Z_{l}^s(F, \M, G) \to \Z_{r}^s(F, \M, G), \qquad \qquad \Xi(M,\sigma) = (M, \sigma^{-1}),
  \end{equation}
  which is the identity on morphisms.
\end{lma}
\begin{proof}
  Suppose that $\sigma$ is an invertible left half-braiding on an object $M\in \M$.
  We show that $\sigma^{-1}$ defines a right half-braiding. Precomposing the equation in \equref{center diag1} by 
	$(\sigma_Y^{-1}\lhd\id_{F(X)})(\id_{G(X)}\rhd\sigma_X^{-1})$ and postcomposing by $\sigma_{YX}^{-1}$ yields 
	the desired compatibility of right half-braidings with the lax functor structures. 
  Analogous calculations show that $\sigma^{-1}$ is compatible with the lax units of $F$ and $G$ and that $\Xi$ sends any morphism 
	in the strong left center to a morphism of the strong right center.
  The proof is concluded by constructing $\Xi^{-1}$ in the same spirit as $\Xi$.
  That is,  by mapping invertible right half-braidings to their inverses.
\qed\end{proof}


\medskip

The construction of (left) strong twisted center categories can be seen as a result of the following composition of 2-functors:
$$\C\x\D\x\Bimod \stackrel{(F,G)}{\to} \E\x\E\x\Bimod \stackrel{\Z_\E}{\to} \Z(\E)\x\Mod$$
$$\hspace{1cm} \M \hspace{0,8cm} \mapsto \hspace{0,6cm} {}_G\M_F \hspace{0,8cm} \mapsto \hspace{0,4cm} \Z_\E({}_G\M_F)$$
where $(F,G)$ denotes precomposing the left and right action by $F$ and $G$, respectively, 
and $\Z_\E$ is defined as in \cite[Section 3.4]{ENO}. The term ``twisted'' is motivated by this composition. Namely, if $F$ and $G$ 
are strong monoidal functors, a $\C\x\D$-bimodule category structure is twisted by them into an $\E$-bimodule structure.

\section{Categorical centers as a data in a tricategory}

At the core of our investigation in this section are (weak) twisted centers and their interpretation from a higher categorical point of view.
We will show that center categories are hom-categories of hom-bicategories of a particular tricategory. 
Namely, the tricategory of bicategories with a single object.

\subsection{Categorical centers as (co)lax natural transformations}

For the interpretation of center categories from the perspective of 2-categories we first recall the definitions of lax 
and colax functors between bicategories and of lax and colax natural transformations between the latter. 

\begin{defn}
  A \emph{lax functor} $(\F, \F^2, \F^0) \colon \K \to \K'$ between 2-categories consists of
  \begin{enumerate}
    \itemsep0em
  \item an assignment $\Ob\K\ni A \mapsto\F(A) \in \Ob\K'$,  
	\item for all $A,B \in \Ob\K$ a local functor $\F_{A,B} \colon \K(A, B) \to \K'(\F(A), \F(B))$,
  \item a natural transformation 
    \begin{equation*}
      \F^2_{g,f} \colon F(g) \otimes' F(f) \Rightarrow F(g \otimes f), \qquad \qquad \text{ for } (g,f)\in \K(B,C)\times\K(A,B),
    \end{equation*}
		and
  \item a natural transformation 
    \begin{equation*}
      \F^0_{A} \colon id_{\F(A)} \Rightarrow \F(id_A), \qquad \qquad \text{ for } A\in\Ob\K,
    \end{equation*}
  \end{enumerate}
  so that $\F^2$ and $\F^0$ satisfy \emph{associativity} and \emph{unitality} laws. \\
	When the natural transformations $\F^2$ and $\F^0$ are directed in the opposite direction 
	and satisfy \emph{coassociativity} and \emph{counitality} laws, one has a {\em colax functor}. 
	One speaks about a {\em pseudofunctor} if $\F^2$ and $\F^0$ are isomorphisms. 	 
\end{defn}

Lax transformations can be defined both for lax and colax functors. The same holds for colax transformations, so that there are four 
variations of definitions, depending on the situation.

\begin{defn} \delabel{colax tr}
  Let $(\F, \F^2, \F^0) \colon\K\to\K'$ and $(\G, \G^2, \G^0)\colon\K\to\K'$ be lax functors between 2-categories.
  A colax natural transformation $\chi\colon\F\Rightarrow\G$ consists of
\begin{enumerate}
\item a 1-cell $\chi_A\colon \F (A) \to \G (A)$ for each object $A\in\Ob\K$, and
\item for every pair of objects $A, B \in Ob\K$ a collection of 2-cells
  \begin{equation} \eqlabel{def-colax}
    \{\chi_f \colon \chi_B\circ \F_{A,B}(f) \Rightarrow \G_{A,B}(f)\circ \chi_A \mid f \in \K(A,B) \}
  \end{equation}
\end{enumerate}
natural in $f$ subject to \emph{colax multiplicativity}
\begin{equation} \label{eq: lax-mult}
\begin{tikzcd}[sep = 1.2em]
	&&&&&&&& {\F(B)} \\
	{\F(A)} && {\F(B)} && {\F(C)} && {\F(A)} && {} && {\F(C)} \\
	&&&&& {=} \\
	{\G(A)} && {\G(B)} && {\G(C)} && {\G(A)} &&&& {\G(C)}
	\arrow["{\F(f)}"{description}, from=2-1, to=2-3]
	\arrow["{\F(g)}"{description}, from=2-3, to=2-5]
	\arrow["{\chi_C}"{description}, from=2-5, to=4-5]
	\arrow["{\G(g)}"{description}, from=4-3, to=4-5]
	\arrow["{\chi_B}"{description}, from=2-3, to=4-3]
	\arrow["{\chi_A}"{description}, from=2-1, to=4-1]
	\arrow["{\G(f)}"{description}, from=4-1, to=4-3]
	\arrow[""{name=0, anchor=center, inner sep=0}, "{\G(gf)}"{description}, curve={height=40pt}, from=4-1, to=4-5]
	\arrow["{\F(f)}"{description}, curve={height=-12pt}, from=2-7, to=1-9]
	\arrow["{\F(g)}"{description}, curve={height=-12pt}, from=1-9, to=2-11]
	\arrow["{\F(gf)}"{description}, from=2-7, to=2-11]
	\arrow["{\G(gf)}"{description}, from=4-7, to=4-11]
	\arrow["{\chi_C}"{description}, from=2-11, to=4-11]
	\arrow["{\chi_A}"{description}, from=2-7, to=4-7]
	\arrow["{\chi_f}"', shorten <=11pt, shorten >=11pt, Rightarrow, from=2-3, to=4-1]
	\arrow["{\chi_g}", shorten <=11pt, shorten >=11pt, Rightarrow, from=2-5, to=4-3]
	\arrow["{\chi_{gf}}"', shorten <=22pt, shorten >=22pt, Rightarrow, from=2-11, to=4-7]
	\arrow[draw=none, from=2-9, to=2-7]
	\arrow["{\F^2}"', shorten <=2pt, shorten >=2pt, Rightarrow, from=1-9, to=2-9]
	\arrow["{\G^2}"', shorten <=3pt, shorten >=3pt, Rightarrow, from=4-3, to=0]
\end{tikzcd}
\end{equation}
and \emph{colax unitality}
\begin{equation} \label{eq:colax-unity}
  \begin{tikzcd}
	{\F(A)} && {\F(A)} && {\F(A)} && {\F(A)} \\
	&&& {=} \\
	{\G(A)} && {\G(A)} && {\G(A)} && {\G(A)}
	\arrow["{\chi_A}"{description}, from=1-1, to=3-1]
	\arrow["{\chi_A}"{description}, from=1-3, to=3-3]
	\arrow[""{name=0, anchor=center, inner sep=0}, "{\id_{\G(A)}}"{description}, curve={height=-18pt}, from=3-1, to=3-3]
	\arrow[""{name=1, anchor=center, inner sep=0}, "{\G(\id_A)}"{description}, curve={height=18pt}, from=3-1, to=3-3]
	\arrow["{\id_{\F(A)}}"{description}, curve={height=-18pt}, from=1-1, to=1-3]
	\arrow[""{name=2, anchor=center, inner sep=0}, "{\id_{\F(A)}}"{description}, curve={height=-18pt}, from=1-5, to=1-7]
	\arrow[""{name=3, anchor=center, inner sep=0}, "{\F(\id_A)}"{description}, curve={height=18pt}, from=1-5, to=1-7]
	\arrow["{\chi_A}"{description}, from=1-7, to=3-7]
	\arrow["{\chi_A}"{description}, from=1-5, to=3-5]
	\arrow["{\G(\id_A)}"{description}, curve={height=18pt}, from=3-5, to=3-7]
	\arrow["\id"', shift right=1, shorten <=11pt, shorten >=11pt, Rightarrow, from=1-3, to=3-1]
	\arrow["{\chi_{\id_A}}", shift left=1, shorten <=11pt, shorten >=11pt, Rightarrow, from=1-7, to=3-5]
	\arrow["{\F^0}"', shorten <=5pt, shorten >=5pt, Rightarrow, from=2, to=3]
	\arrow["{\G^0}"', shorten <=5pt, shorten >=5pt, Rightarrow, from=0, to=1]
\end{tikzcd}
\end{equation}
If the 2-cells of $\chi$ are invertible, it  is called a \emph{pseudonatural transformation}. 
In case they are identities, one speaks of a \emph{strict natural transformation}.
\end{defn}

By reverting the direction of the 2-cells of $\chi$ one obtains the notion of a {\em lax natural transformation} between lax functors.

\smallskip

We start by a simple observation that entails a marvelous fact. 

\begin{prop}\prlabel{Shim}
  Let $F, G \colon \C \to  \D$ be lax monoidal functors. The objects of the weak twisted center $\Z_l^w(F,\D,G)$ are canonically in bijection with colax natural transformations $\chi \colon Del(F) \Rightarrow Del(G)$ between the induced lax functors $Del(F), Del(G): Del(\C)\to Del(\D)$.
  Under this identification, the objects of the strong center correspond to pseudonatural transformations.
\end{prop}

\begin{proof}
Since both bicategories $Del(\C), Del(\D)$ have a single object, there is a single 1-cell component of $\chi \colon Del(F) \Rightarrow Del(G)$, which is 
a distinguished object $D_{\chi}=\chi_*$ in $\D$. The 2-cell components of $\chi$ amount to 
morphisms $\chi_X: D_{\chi}\ot F(X)\to G(X)\ot D_{\chi}$ in $\D$ natural in $X$ (mind that we do not flip the order of factors when translating from $Del(\C)$ to $\D$, as we assume the reversed tensor product in the sense of \rmref{reversed}), and the 
colax multiplicativity and unity translate into the commuting diagrams \equref{center diag1} and \equref{center diag2}. The second claim is immediate. 
\qed\end{proof}

\begin{rem} \rmlabel{co/lax}
  An analogous statement to the previous proposition for right half-braidings can be obtained by considering lax instead of colax natural transformations between lax functors.
\end{rem}


To obtain a bicategorical interpretation of the morphisms in a center category, we need to recall the definition of modifications. 

\begin{defn}
A modification $a: \chi\Rrightarrow\psi$ between two colax natural transformations $\chi, \psi:\F\Rightarrow\G:\K\to\K'$ consists of a family of 2-cells $a_A:\chi(A)\Rightarrow\psi(A)$, indexed by the objects $A\in\Ob\K$, such that for every 1-cell $f\in\K(A,B)$ we have:
\begin{equation} \label{eq:modification-condition}
\begin{tikzcd}[sep=1.2em]
	{\F(A)} &&& {\G(A)} && {\F(A)} &&& {\G(A)} \\
	&&&& {=} \\
	{\F(B)} &&& {\G(B)} && {\F(B)} &&& {\G(B)}
	\arrow[""{name=0, anchor=center, inner sep=0}, "{\chi_A}"{description}, from=1-1, to=1-4]
	\arrow["{F(f)}"{description}, from=1-1, to=3-1]
	\arrow["{G(f)}"{description}, from=1-4, to=3-4]
	\arrow["{\chi_B}"{description}, from=3-1, to=3-4]
	\arrow[""{name=1, anchor=center, inner sep=0}, "{\psi_A}"{description}, curve={height=-30pt}, from=1-1, to=1-4]
	\arrow["{\chi_f}", shorten <=16pt, shorten >=16pt, Rightarrow, from=3-1, to=1-4]
	\arrow["{\psi_A}"{description}, from=1-6, to=1-9]
	\arrow["{G(f)}"{description}, from=1-9, to=3-9]
	\arrow["{F(f)}"{description}, from=1-6, to=3-6]
	\arrow[""{name=2, anchor=center, inner sep=0}, "{\psi_B}"{description}, from=3-6, to=3-9]
	\arrow[""{name=3, anchor=center, inner sep=0}, "{\chi_B}"{description}, curve={height=30pt}, from=3-6, to=3-9]
	\arrow["{\psi_f}", shorten <=16pt, shorten >=16pt, Rightarrow, from=3-6, to=1-9]
	\arrow["{a_A}", shorten <=4pt, shorten >=4pt, Rightarrow, from=0, to=1]
	\arrow["{a_B}", shorten <=4pt, shorten >=4pt, Rightarrow, from=3, to=2]
\end{tikzcd}
\end{equation}
\end{defn}

For two lax functors $\F,\G:\B\to\B'$ among bicategories let $\text{Colax}(\F, \G)$ and $\text{Lax}(\F, \G)$ denote the 
categories of colax (respectively lax) natural transformations and their modifications.
Similarly, $\text{Pseudo}(F, G)$ denotes the category of pseudonatural transformations and their modifications.

\begin{prop} \prlabel{isos}
  Let $F, G \colon \C \to \D$ be lax monoidal functors. There are canonical isomorphisms of categories: 
$$\Z_l^w(F,\D, G)\iso\text{Colax}(Del(F),Del(G)),\qquad\qquad \Z_l^s(F,\D,G)\iso\text{Pseudo}(Del(F),Del(G)),$$
$$\Z_r^w(F,\D,G)\iso\text{Lax}(Del(F),Del(G)),\qquad\qquad \Z_r^s(F,\D,G)\iso\text{Pseudo}(Del(F),Del(G)).$$
\end{prop}

\begin{proof}
We only prove the claims for left center categories, as the other cases are analogous. 
  Let $\chi, \psi \colon Del(F) \Rightarrow Del(G)$ be colax natural transformations and write $(D_\chi, \chi), (E_\psi, \psi) 
	\in \Z_w^l(F,\D,G)$ for their corresponding objects in the weak left center.
  Since $Del(\C)$ and $Del(\C)$ are deloopings of monoidal categories, any modification $a\colon \chi\Rrightarrow\psi$ is defined by a single 
	morphism $f\colon D_{\chi}\to E_{\psi}$ satisfying for all $X\in\C$ the following identity:
 \begin{equation}  \eqlabel{morph}
\gbeg{2}{6}
\got{1}{D} \got{1}{F(X)} \gnl
\gcl{1} \gcl{1} \gnl
\glmptb \gnot{\hspace{-0,34cm}\chi_X} \grmptb \gnl
\gcl{1} \gbmp{f} \gnl
\gcl{1} \gcl{1} \gnl
\gob{1}{G(X)} \gob{1}{E}
\gend=
\gbeg{2}{6}
\got{1}{D} \got{1}{F(X)} \gnl
\gcl{1} \gcl{1} \gnl
\gbmp{f} \gcl{1} \gnl
\glmptb \gnot{\hspace{-0,34cm}\psi_X} \grmptb \gnl
\gcl{1} \gcl{1} \gnl
\gob{1}{G(X)} \gob{1}{E.}
\gend
\end{equation}
This is precisely the defining equation of a morphism in the weak center and the claim follows.
\qed\end{proof}

\bigskip



In the above proposition we started from two lax monoidal functors between monoidal categories to obtain the result. One can also start 
from two lax functors between 2-categories in a specific way to obtain an analogous result. For this purpose recall the interplay 
between ordinary categories and bicategories encoded in the delooping isomorphism \equref{eq:delooping} and the discussion after \rmref{reversed}.

\begin{prop}  \prlabel{Naidu-gen}
Let $\E$ be a monoidal category, $\K$ a bicategory which has at least two objects $0$ and $1$, and assume that $\F,\G:Del(\E)\to\K$ 
are two lax functors such that $\F(*)=0$ and $\G(*)=1$. There are canonical isomorphisms of categories 
$$\Z_l^w(F,\M, G)\iso\text{Colax}(\F,\G),\qquad\qquad \Z_l^s(F,\M,G)\iso\text{Pseudo}(\F,\G),$$
$$\Z_r^w(F,\M, G)\iso\text{Lax}(\F,\G),\qquad\qquad \Z_r^s(F,\M,G)\iso\text{Pseudo}(\F,\G),$$
for a suitable bimodule category $\M$ 
and lax monoidal functors $F$ and $G$. 
\end{prop}

\begin{proof} 
We set $\D:=\K(0,0), \C:=\K(1,1)$ and $\M:=\K(0,1)$. Then the two  lax functors yield lax monoidal functors $F=\F_{*,*}:\E\to\D, G=\G_{*,*}:\E\to\C$. 
The unique 1-cell component of a colax natural transformation $\chi:\F\to\G$ is an object $M$ living in the $(\C\x\D)$-bimodule category $\M$, 
as it is a 1-cell mapping $0=\F(*)\to\G(*)=1$ in $\K$. The 2-cell component yields a half-braiding given by morphisms 
$\chi_X: M\lhd F(X)\to G(X)\rhd M$ in $\M$, for every 1-endocell in $\K$, {\em i.e.} an object $X\in\E$. The rest follows 
as in \prref{isos}. 
\qed\end{proof}

Recall that we recovered the left center category from \cite{GNN} as $\Z^l_\C(\M)\coloneqq \Z^s_l(\Id,\M,\Id)$, and similarly the right one, 
and that they are isomorphic by \leref{left-iso-right}.

\begin{cor} \colabel{Naidu}
Let $\C$ be a monoidal category, $\K$ a bicategory which has at least two objects $0$ and $1$, and assume that $\F,\G:Del(\C)\to\K$ 
are two pseudofunctors such that $\F(*)=0,\G(*)=1$ and $\F_{*,*}=\G_{*,*}=\Id_\C$. There are canonical isomorphisms of categories 
$$\Z^l_\C(\M)\iso\Pseudo(\F,\G)\iso\Z^r_\C(\M)$$
where $\M=\K(0,1)$. 
\end{cor}



\subsection{The bicategory of center categories}

The bicategorical perspective gives us a deeper insight of ``why'' the twisted center categories can be composed between each other, 
and in particular ``why'' ``$F\x F$-twisted'' center categories are monoidal. Namely, for fixed bicategories $\B$ and $\B'$ there are bicategories 
$\Lax_{lx}(\B,\B')$ and $\Lax_{clx}(\B,\B')$ of lax functors $\B \to \B'$, lax (resp. colax) transformations and their modifications, see~\cite[Theorem 4.4.11]{JY}.
The composition of (co)lax natural transformations (as 1-cells in these bicategories) 
is given by what is known as the vertical composition of such transformations. The hom-categories 
are $\text{Lax}(\F, \G)$ and $\text{Colax}(\F, \G)$ for two  lax functors $\F,\G$, respectively. The horizontal composition 
in the two bicategories corresponds to the composition of the right, respectively left, weak twisted center categories between themselves. 
In particular, the fact that $\Lax(\F,\F)$ and $\text{Colax}(\F, \F)$ are monoidal entails that the ``$F\x F$-twisted'' 
center categories are monoidal as well. We explain now how this is achieved.

Let us fix a monoidal category $\E$ and a 2-category $\K$ of which we think as a collection of monoidal categories and composable bimodule categories between them. According to \prref{Naidu-gen}, we can identify the hom-categories of $\Lax_{clx}(Del(\E),\K)$ with the weak left centers relative to lax monoidal functors from $\E$ to the  endo-1-categories of $\K$.
In order to emphasize this point of view, we write $\Z_l^w(\E,\K) \coloneqq \Lax_{clx}(Del(\E),\K)$. 
Given lax functors $\F, \G \colon Del(\E) \to \K$, we write $\Z_l^w(\F,\K,\G)$ for the hom-category of $\Z_l^w(\E, \K)$ with source $\F$ 
and target $\G$. (Observe that they are categories $\Z_l^w(F,\M,G)$ from \prref{Naidu-gen}, where $F,G$ are lax monoidal functors 
induced by $\F,\G$ and $\M=\K(0,1)$.)

The above observation readily implies that (weak left) centers can be organized into a bicategory. We refer to $\Z_l^w(\E, \K)$ as the 
\emph{bicategory of weak left centers of $\K$ relative to $\E$}. Our next result translates its horizontal composition into a 
`pre-tensor product' between weak relative centers. Given two vertically composable 2-cells $x\stackrel{\alpha}{\Rightarrow} y\stackrel{\beta}{\Rightarrow}z: 
A\to B$, we denote their vertical composition in equations by the fraction $\frac{\alpha}{\beta}$.

\begin{prop} \prlabel{center-2-cat}
  Let $\E$ be a monoidal category and $\K$ a 2-category. The weak left centers of $\K$ relative  to $\E$ form a bicategory 
	$\Z_l^w(\E, \K)$. For all lax functors  $\F, \G, \HH: Del(\E)\to\K$, 
	its horizontal composition is given by 
  \begin{equation} \eqlabel{eq:}
    \begin{aligned}
      \circ_{\F,\G,\HH} \colon \Z_l^w(\G, \K, \HH) \times \Z_l^w(\F, \K, \G) &\to \Z_l^w(\F, \K, \HH) \\[0.2\baselineskip]
      \left((N, \tau), (M, \sigma) \right) &\mapsto (N\circ M, \frac{\id_N \circ\sigma}{\tau \circ\id_M}) \\
      (g,f) &\mapsto (g\circ f)
    \end{aligned}
  \end{equation}  
\end{prop}

\begin{proof}
  The first claim is merely a recapitulation that $\Z_l^w(\E,\K) = \Lax_{clx}(Del(\E),\K)$ is a bicategory.
  By \prref{Naidu-gen}, the objects and morphisms of any weak center are in correspondence with appropriate colax natural 
	transformations and modifications. 
  Applying this identification, one immediately obtains the above formulas from their respective composition rules. 
	(Observe that $M\in{}_\C\M_\D=\K(0,1)$ and $N\in{}_\Pp\N_\C=\K(1,2)$, where $\F(*)=0,\G(*)=1$ and $\HH(*)=2$, and 
	 $\D:=\K(0,0), \C:=\K(1,1)$ and $\Pp:=\K(2,2)$, in the sense of the proof of \prref{Naidu-gen}.) 
\qed\end{proof}

Analogous considerations hold for the bicategory of weak right centers $\Z_r^w(\E,\K) \coloneqq \Lax_{lx}(Del(\E),\K)$. 

The consequences of the above result are best exemplified by considering the bicategory $\Z_l^w(\E, \E)$ of all twisted weak left centers 
$\Z_l^w(F,\E,G)$ of $\E$. The weak left analogue $\Z_l^w(\E)$ of the Drinfel`d center corresponds to the endo-category on the identity functor of $\E$.
The previous result recovers its monoidal structure.
Moreover, it implies that, if only the left or right action of $\E$ on itself is twisted, the resulting center canonically becomes a right, respectively, left module category over $\Z_l^w(\E)$.
This gives a theoretical justification to the constructions of \cite{HZ} involving the anti-Drinfel`d center and its opposite.

\medskip

Let us consider the pseudo-pseudo version of the bicategories $\Lax_{lx}(\B,\B')$ and $\Lax_{clx}(\B,\B')$. It is a 
bicategory that we denote by $\Ps_{ps}(\B,\B')$ with hom-categories $\Ps(\F,\G)$ for pseudofunctors $\F,\G:\B\to\B'$. 
In the particular case when $\F=\G=\Id_\B$ we have the center category $\Z(\B)$ of the bicategory $\B$ introduced in 
\cite{Ehud}. Our above considerations allude to the possibility to consider ``twisted center categories of the bicategory $\B'$''. 
In this case the twisting is done through the pseudo- (or lax or colax) functors $\B\to\B'$.

\subsection{The tricategory that encomapsses strong center categories}

Let $\Bicat$ denote the tricategory of bicategories, pseudofunctors, pseudonatural transformations, and modifications.

\begin{rem}  \rmlabel{tricat-bicat}
Although it is sufficient to consider only lax functors and colax (resp. lax) natural transformations in order to recover left 
(resp. right) weak twisted center category, in order to form a tricategory whose 2- and 3-cells are some kind of bicategorical functors 
and transformations, 
both kinds of cells should be of {\em pseudo} type, {\em i.e.} they both should have isomorphisms for their respective defining structures. 
Namely, on one hand, in order to be able to define the horizontal composition of two lax (or colax) natural transformations, both lax and 
colax structures of the functors they act to are needed (see {\em e.g.} Second problem in \cite{Lack-Icons}). 
On the other hand, one sees from the diagram (11.3.12) of \cite[Chapter 11.3]{JY} that in order to construct the isomorphism 
interchange 3-cell one needs both pseudofunctors and pseudonatural transformations.
\end{rem}

One clearly has:

\begin{prop} 
 By delooping, the full sub-tricategory of $\Bicat$ whose objects are bicategories with a single object gives rise to a tricategory 
$\Bicat^*$ with monoidal categories as objects, strong monoidal functors as 1-cells, and for each pair $F, G\colon \C \to \D$ 
of such functors the strong center $\Z_{l}^{s}(F,\D,G)\iso\Z_{r}^{s}(F,\D, G)$ as hom-category.
\end{prop}

In view of \rmref{tricat-bicat} it becomes clear why we do not speak of a tricategory that contains weak center categories as 
bottom hom-categories. 

\medskip

By further restricting the tricategory $\Bicat^*$ to finite tensor categories (in the sense of \cite{EGNO}), one has that every hom-bicategory 
$\Bicat^*(\C,\D)$ is precisely the bicategory $\TF(\C,\D)$ from \cite[Section 3.2]{Shim}, of tensor functors between finite tensor categories $\C$ and $\D$. 
As the author works in a context of rigid categories, he shows that for every hom-category $Z_l^w(F,G)$ of the bicategory 
$\TF(\C,\D)$ and every object $(V,\sigma_V)\in Z_l^w(F,G)$ the transformation $\sigma_V$ is invertible (\cite[Lemma 3.1]{Shim}), and 
$(V,\sigma_V)$ has a left and a right dual object in $Z_l^w(G,F)$. We will generalize this to 2-categories in \prref{w=s}. 


\medskip

{\bf Question.} In \prref{Naidu-gen} we dealt with lax functors $\F,\G:Del(\E)\to\K'$ from a one-object bicategory. 
If we allow for the domain bicategory 
$\K$ to have more than one object, we wonder what the categories of (co)lax transformations $\chi: \F\Rightarrow\G$ recover. In particular, 
what is the ``meaning in nature'' of the 2-cell components $\chi_X: \chi(B)\circ F(X)\to G(X)\circ \chi(A)$ for 1-cells $X:A\to B$ in $\K$?

\subsection{String diagrams in 2-categories}

In \ssref{adj} and throughout \seref{Bilax} and \seref{2-cat Bilax} we will use string diagrams for 2-categories
(again relying on the biequivalence of any bicategory with a 2-category).
Our string diagrams are read from top to bottom and (in the context of 2-categories) from right to left. The domains and codomains
of the strings stand for 1-cells, while the strings themselves and boxes stand for 2-cells. The 0-cells are to be understood from the context
(reading the 1-cells from right to left). Observe that such string diagrams which depict 2-cells in a 2-category $\K$ 
acting on the same underlying 0-cells $A\in\Ob\K$ (that is, morphisms in the monoidal categories $\K(A,A)$ for every $A$) correspond exactly 
to the string diagrams in the monoidal categories $\K(A,A)$ (due to the isomorphism \equref{eq:delooping}). This is even more clear if 
$\K=Del(\C)$ for some monoidal category $\C$. 

\bigskip

Let $\F$ be a lax functor and $\G$ a colax functor. We depict their lax, respectively colax structures by diagrams in the following way:
$$
\gbeg{3}{3}
\got{1}{\F(g)} \got{3}{\F(f)} \gnl
\gwmu{3} \gnl
\gob{3}{\F(gf)}
\gend \qquad 
\gbeg{3}{3}
\got{1}{id_{\F(A)}} \gnl
\gu{1} \gnl
\gob{1}{\F(id_A)}
\gend  \qquad 
\gbeg{3}{3}
\got{3}{\G(gf)} \gnl
\gwcm{3} \gnl
\gob{1}{\G(g)} \gob{3}{\G(f)} \gnl
\gend  \qquad 
\gbeg{3}{3}
\got{1}{\G(id_A)} \gnl
\gcu{1} \gnl
\gob{1}{id_{\F(A)}}
\gend  
$$
where $g,f$ are composable 1-cells and $A$ a 0-cell in the domain 2-category. We will often simplify the notation $\circ$ 
for the composition of 1-cells by concatenation.

\bigskip

Observe that a colax transformation between two lax functors, \deref{colax tr}, is nothing but a distributive law between lax functor
structures that is moreover natural in 1-cells.
Similarly, a colax transformation between two colax functors is a distributive law between colax functor structures that is moreover natural in 1-cells.
In string diagrams we may write the latter as follows:
\begin{equation}\eqlabel{phi-colax}
\gbeg{3}{5}
\got{1}{\chi_C}\got{2}{\F(gf)} \gnl
\gcl{1} \gcmu \gnl
\glmptb \gnot{\hspace{-0,34cm}\chi_g} \grmptb \gcl{1} \gnl
\gcn{1}{1}{1}{-1} \glmptb \gnot{\hspace{-0,34cm}\chi_f} \grmptb \gnl
\gob{-1}{\G(g)} \gvac{2} \gob{1}{\G(f)} \gob{1}{\chi_A}
\gend=
\gbeg{3}{5}
\got{1}{\chi_C}\got{2}{\F(gf)} \gnl
\gcn{1}{1}{2}{2} \gcn{1}{1}{2}{2} \gnl
\gvac{1} \hspace{-0,34cm} \glmpt \gnot{\hspace{-0,34cm}\chi_{gf}} \grmptb \gnl
\gvac{1} \hspace{-0,2cm} \gcmu \gcn{1}{1}{0}{1} \gnl
\gvac{1} \gob{0}{\G(g)} \gvac{1} \gob{1}{\G(f)} \gob{1}{\chi_A}
\gend;
\quad
\gbeg{3}{5}
\got{1}{\chi_A}\got{2}{\F(id_A)}\gnl
\gcl{1} \gcl{1} \gnl
\glmptb \gnot{\hspace{-0,34cm}\chi_{id_A}} \grmptb \gnl
\gcu{1} \gcl{1} \gnl
\gob{3}{\chi_A}
\gend=
\gbeg{3}{5}
\got{1}{\chi_A}\got{2}{\F(id_A)}\gnl
\gcl{1} \gcl{1} \gnl
\gcl{2}  \gcu{1} \gnl
\gob{1}{\chi_A}
\gend;
\quad
\gbeg{3}{6}
\got{1}{\chi_B}\got{1}{\F(x)} \gnl
\gcl{1} \gcl{1} \gnl
\gcl{1} \glmptb \gnot{\hspace{-0,34cm}\F(\alpha)} \grmp \gnl
\glmptb \gnot{\hspace{-0,34cm}\chi_y} \grmptb \gnl
\gcl{1} \gcl{1} \gnl
\gob{1}{G(y)} \gob{2}{\chi_A}
\gend=
\gbeg{2}{6}
\gvac{1} \got{1}{\chi_B}\got{1}{\F(x)} \gnl
\gvac{1} \gcl{1} \gcl{1} \gnl
\gvac{1} \glmptb \gnot{\hspace{-0,34cm}\chi_x} \grmptb \gnl
\glmp \gnot{\hspace{-0,34cm}\G(\alpha)} \grmptb\gcl{1} \gnl
\gvac{1} \gcl{1} \gcl{1} \gnl
\gvac{1} \gob{0}{G(y)} \gob{3}{\chi_A}
\gend
\end{equation}
for any 2-cell $\alpha:x\Rightarrow y:A\to B$.
(A colax transformation between two lax functors is determined by string diagrams that are both vertically and horizontally symmetric
to the above ones.)
Dually (taking opposite 2-cells, {\em i.e.} taking vertically symmetric diagrams to those in \equref{phi-colax}) we obtain string diagrams
for lax transformations between two lax functors.

\subsection{Adjoints in 2-categories} \sslabel{adj}

Dualisability plays a prominent role in the study of monoidal categories and the closely related subject of (extended) topological quantum field theories, see \cite{BD}. For example, Section 2 of \cite{DSS} and Section 4 of \cite{HZ} discuss and utilize `duals' of bimodule categories and centers. 
In the following, we want to provide a 2-categorical perspective on these constructions, thereby giving a theoretical underpinning for some of the ad-hoc constructions of \cite{HZ}.

We start by briefly recalling the notion of adjoint 1-cells in 2-categories, see for example \cite{Gray, Lack1}. 
Let $f \colon A\to B$ be a 1-cell in a 2-category $\K$.
A \emph{left adjoint} of $f$ is a 1-cell $u \colon B \to A$ together with two 2-cells $\eta: \id_A\to u f$ and $\varepsilon: f u \to \id_B$ such that
\begin{equation*}
  \frac{\Id_f\circ \eta}{\varepsilon\circ\Id_f} = \id_f
  \qquad \qquad \text{and} \qquad \qquad
   \frac{\eta\circ\Id_u}{\Id_u \circ\varepsilon} = \id_u.
\end{equation*}
Similarly, a \emph{right adjoint} of $f$ is a 1-cell $v \colon B \to A $ together with two 2-cells $ \bar \eta \colon \id_B \to f v$ and $\bar \varepsilon \colon v f \to \id_A$ such that
\begin{equation*}
  \frac{\Id_v \circ \bar \eta}{\bar \varepsilon\circ\Id_v} = \id_f
  \qquad \qquad \text{and} \qquad \qquad
   \frac{\bar \eta \circ\Id_f}{f \circ \bar \varepsilon} = \id_f.
\end{equation*}

In string diagrams we will write $\eta=
\gbeg{2}{1}
\gdb \gnl
\gend$ and $\Epsilon=
\gbeg{2}{1}
\gev \gnl
\gend$, and they satisfy the laws:
$$\gbeg{3}{4}
\got{1}{} \got{3}{u} \gnl
\gdb  \gcl{1} \gnl
\gcl{1} \gev \gnl
\gob{1}{u}
\gend=\Id_u
\qquad\textnormal{and}\qquad 
\gbeg{3}{4}
\got{1}{f} \got{1}{} \gnl
\gcl{1} \gdb \gnl
\gev \gcl{1} \gnl
\gob{1}{} \gob{4}{\hspace{-0,1cm}f}
\gend=\Id_f.
$$
Pseudofunctors $\F:\K\to\K'$ preserve (and reflect) adjoints and we have:
\begin{equation} \eqlabel{ev-coev on F}
\gbeg{3}{2}
\got{1}{\F(f)} \got{3}{\F(u)} \gnl
\gwev{3} \gnl
\gend
=
\gbeg{3}{4}
\got{1}{\F(f)} \got{3}{\F(u)} \gnl
\gwmu{3} \gnl
\glmp \gnot{\hspace{-0,34cm}\F(\Epsilon)} \grmptb \gnl
\gvac{1} \gcu{1} \gnl
\gend,
\qquad
\gbeg{3}{2}
\gwdb{3} \gnl
\gob{1}{\F(u)} \gob{3}{\F(f)} \gnl
\gend
=
\gbeg{3}{4}
\gvac{1} \gu{1} \gnl
\glmp \gnot{\hspace{-0,34cm}\F(\eta)} \grmptb \gnl
\gwcm{3} \gnl
\gob{1}{\F(u)} \gob{3}{\F(f).} \gnl
\gend
\end{equation}


\bigskip

In the aforementioned \cite{HZ}, certain trace-like morphisms were considered in order to implement pivotal structures on the Drinfel`d 
centers. This involves a `lift' of the notion of duals, i.e. adjoints, to the setting of centers of bimodule categories.
With our interpretation of $\Z_l^w(\E,\K) = \Lax_{clx}(Del(\E),\K)$ in \prref{center-2-cat} as a bicategory, for a monoidal category $\E$ 
and a bicategory $\K$, we can now derive a more conceptual version of this construction. 

\begin{defn}
  We refer to a 2-category $\K$ as \emph{autonomous} if all 1-cells in $\K$ have left and right adjoints.
\end{defn}

The most prominent example of an autonomous 2-category is given by the delooping $Del(\E)$ of an autonomous (rigid) 
monoidal category $\E$. We will prove at the end of this section that the full sub-bicategory $\Z_l^{w-ps}(\E, \K)$ 
whose objects are pseudofunctors is autonomous. We will use the following fact, which is immediately proved: 

\begin{lma} \lelabel{lax-ps}
Let $\F,\G:\K\to\K'$ be pseudofunctors between 2-categories, and suppose that $\chi:\F\Rightarrow\G$ is a colax natural transformation 
with respect to their lax functors structures. Then $\chi$ is a colax natural transformation of pseudofunctors. 
\end{lma}

The following result is a 2-categorical interpretation of the fact that the half-braidings over autonomous monoidal 
categories are automatically invertible, see \cite[Lemma 3.1]{Shim}. 

\begin{prop} \prlabel{w=s}
  Let $\K$ be an autonomous 2-category and $\E$  an autonomous monoidal category. For any pair of pseudofunctors 
	$\F, \G \colon Del(\E) \to \K$, 
	the weak and strong centers coincide. Moreover, the inverse of a left half-braiding is a right half-braiding, that is: 
	$\Z_l^{w-ps}(\F, \K, \G) = \Z_l^{s-ps}(\F, \K, \G) \iso \Z_r^{s-ps}(\F, \K, \G) = \Z_r^{w-ps}(\F, \K, \G) $. 
\end{prop}

\begin{proof}\label{prop:adjoints-imply-inverses}
 For the first claim it suffices to show that the component 2-cells of any colax natural transformation between pseudofunctors $\chi \colon \F \Rightarrow \G$ are invertible.
  Hereto, we fix an object $X\in \E$ and write $X^{\ast}$ for its left adjoint. 
Starting from the left hand-side below and applying the left equality in \equref{ev-coev on F}, 
the coherence of $\chi$ with the functor's 
multiplicativity, naturality of $\chi$ with respect to $\Epsilon_X$ and the coherence of $\chi$  with respect to the functor's counitalities 
(mind \leref{lax-ps}) we reach the equality with the right hand-side:
$$
\gbeg{4}{5}
\got{0}{D_{\chi}} \got{2}{\hspace{0,2cm}\F(X)} \got{2}{\F(X^*)} \gnl
\glmptb \gnot{\hspace{-0,34cm}\chi_X} \grmptb \gcl{1} \gnl
\gcl{1} \glmptb \gnot{\hspace{-0,34cm}\chi_{X^*}} \grmptb \gnl
\gev \gcl{1} \gnl
\gvac{2} \gob{1}{D_{\chi}}
\gend
=
\gbeg{3}{5}
\got{1}{D_{\chi}}\got{2}{\F(X)} \got{2}{\F(X^*)} \gnl
\gcl{3} \gwev{3} \gnl
\gob{1}{D_{\chi}.}
\gend
$$
Then the 2-cell:
\begin{equation*}
\gamma_X:=
\gbeg{4}{5}
\got{0}{\G(X)} \got{3}{D_{\chi}} \gnl
\gcl{2} \gcl{1} \gdb \gnl
\gvac{1} \glmptb \gnot{\hspace{-0,34cm}\chi_{X^*}} \grmptb \gcl{2} \gnl
\gev \gcl{1} \gnl
\gvac{2} \gob{1}{D_{\chi}} \gob{2}{\F(X)}  
\gend  
\end{equation*}
is clearly a left inverse of $\chi_X: D_{\chi} \circ \F(X) \to \G(X) \circ D_{\chi}$. The analogous reasoning, but this time using the coherence 
of $\chi$ with the comultiplicativity and unitality of the functors, shows that $\gamma_X$ is also a right inverse of $\chi_X$. 
The last statement is proved directly. 
\qed\end{proof}

\medskip

Our next result states that adjoints can be `lifted' to weak centers.

\begin{prop} \prlabel{auton}
  Suppose $\E$ to be an autonomous monoidal category and $\K$ to be an autonomous 2-category.
  Then the full sub-bicategory $\Z^{w\x ps}_l(\E, \K) \subset \Z_l^w(\E, \K)$ whose objects are pseudofunctors is autonomous.
\end{prop}

\begin{proof}  
  We only show that any 1-cell in $\Z^{w\x ps}_l(\E, \K)$, that is, an object 
	$(M,\chi) \in \Z^{w\x ps}_l(\F, \K,\G)$, has a right adjoint living in $\Z^{w\x ps}_l(\G, \K,\F)$.
  The case of left adjoints is analogous.
  Due to \prref{w=s}, the half-braiding $\chi$, {\em i.e.} the 2-cells $\chi_X: M \circ \F(X) \to \G(X)\circ M$, is invertible.
  We utilize this fact, to define a  half-braiding on ${}^{\ast}M$ as shown in the next diagram
  \begin{equation*} 
\sigma_X:=
\gbeg{4}{5}
\got{0}{{}^{\ast}M} \got{3}{\G(X)} \gnl
\gcl{2} \gcl{1} \gdb \gnl
\gvac{1} \glmptb \gnot{\hspace{-0,34cm}\chi_X^{-1}} \grmptb \gcl{2} \gnl
\gev \gcl{1} \gnl
\gvac{2} \gob{1}{\F(X)} \gob{2}{{}^{\ast}M.}  
\gend
  \end{equation*}
  A direct computation shows that it satisfies the Equations (2) and (3). 
  For example, we have
$$
\gbeg{7}{7}
\got{0}{{}^{\ast}M} \got{3}{\G(Y)} \got{3}{\G(X)} \gnl
\gcl{2} \gcl{1} \gdb \gcl{2} \gnl
\gvac{1} \glmptb \gnot{\hspace{-0,34cm}\chi_Y^{-1}} \grmptb \gcl{2} \gvac{1} \gdb \gnl
\gev \gcl{2} \gvac{1} \glmptb \gnot{\hspace{-0,34cm}\chi_X^{-1}} \grmptb \gcl{3} \gnl
\gvac{3} \gev \gcl{1} \gnl
\gvac{2} \gwmu{4} \gnl
\gvac{3} \gob{1}{\F(Y\ot X)} \gvac{2} \gob{1}{{}^{\ast}M} 
\gend
=
\gbeg{5}{6}
\got{0}{{}^{\ast}M} \got{3}{\G(Y)} \got{0}{\G(X)} \gnl
\gcl{1} \gcl{1} \gcl{1} \gdb \gnl
\gcl{2} \gcl{1} \glmptb \gnot{\hspace{-0,34cm}\chi_X^{-1}} \grmptb \gcl{2} \gnl
\gvac{1} \glmptb \gnot{\hspace{-0,34cm}\chi_Y^{-1}} \grmptb \gcl{1} \gnl
\gev \gmu  \gcl{1} \gnl
\gvac{2} \gob{1}{\F(Y\ot X)}  \gob{3}{{}^{\ast}M}
\gend
=
\gbeg{6}{6}
\got{0}{{}^{\ast}M} \got{3}{\G(Y)} \got{0}{\G(X)} \gnl
\gcl{1} \gwmu{3} \gnl
\gcn{2}{2}{1}{3} \gcl{1} \gdb \gnl 
\gvac{2} \glmptb \gnot{\hspace{-0,34cm}\chi_{Y\ot X}^{-1}} \grmptb \gcl{2} \gnl
\gvac{1} \gev \gcl{1} \gnl
\gvac{3} \gob{0}{\F(Y\ot X)} \gob{4}{{}^{\ast}M.}  
\gend
$$
To conclude the proof, we show that the unit $\eta \colon \id \to M \circ {}^{\ast}M$ and counit $\varepsilon \colon 
{}^{\ast}M \circ M \to \id$ lift to 
morphisms in the center category $\Z^{w\x ps}_l(\G,\K,\F)$. For the counit this follows from the computation depicted below.
$$
\gbeg{5}{7}
\got{0}{{}^{\ast}M} \got{3}{M} \got{0}{\F(X)} \gnl
\gcl{4}  \glmptb \gnot{\hspace{-0,34cm}\chi_X} \grmptb \gnl
\gvac{1} \gcl{2} \gcn{1}{1}{1}{5}  \gnl
\gvac{2} \gdb \gcl{2} \gnl
\gvac{1} \glmptb \gnot{\hspace{-0,34cm}\chi_X^{-1}} \grmptb \gcl{1} \gnl
\gev \gcl{1} \gev \gnl
\gvac{2} \gob{1}{\F(X)}  
\gend
=
\gbeg{4}{5}
\got{0}{{}^{\ast}M} \got{3}{M} \got{0}{\F(X)} \gnl
\gcl{2}  \glmptb \gnot{\hspace{-0,34cm}\chi_X} \grmptb \gnl
\gvac{1} \glmptb \gnot{\hspace{-0,34cm}\chi_X^{-1}} \grmptb \gnl
\gev \gcl{1} \gnl
\gvac{2} \gob{1}{\F(X)}  
\gend
=
\gbeg{4}{5}
\got{0}{{}^{\ast}M} \got{3}{M} \got{0}{\F(X)} \gnl
\gev \gcl{3} \gnl
\gvac{2} \gob{1}{\F(X)}  
\gend 
$$
An analogous argument shows that the unit also becomes a morphism in the center.
\end{proof}

\section{Bilax functors}  \selabel{Bilax} 

We are interested in functors on bicategories that are both lax and colax but not necessarily pseudofunctors.
Likewise, we are interested in natural transformations that are both lax and colax but not necessarily pseudonatural
transformations, as well as in their modifications. In particular, we introduce the notions of a bilax functor, bilax 
natural transformation and bilax modifications. We formulate them for 2-categories, just to avoid the use of associators and unitors,
but the corresponding definitions for bicategories can be formulated in a straightforward fashion.
To that end, we fix 2-categories $\K$ and $\K'$. In this section we introduce bilax functors, and leave the resting two 
notions for the next section.

\subsection{Bilax functors}




We reiterate that we will often simplify the notation $\circ$ for the horizontal composition of 1- and 2-cells by concatenation.

\begin{defn} 
Let $\F:\K\to\K'$ be a 2-functor. A {\em Yang-Baxter operator} for $\F$ consists of 
a collection of 2-cells $\nu_{g,f}:\F(g)\F(f)\Rightarrow\F(f)\F(g)$ in $\K'$, natural in 1-endocells $f,g$ of $\K$, 
which satisfy
the Yang-Baxter equation \vspace{-0,2cm}
\begin{equation}\eqlabel{YBE F}
\gbeg{4}{5}
\got{1}{\F(h)} \got{2}{\F(g)} \got{1}{\F(f)} \gnl
\glmptb \gnot{\hspace{-0,34cm}\nu_{h,g}} \grmptb \gcn{1}{1}{3}{1} \gnl
\gcl{1} \glmptb \gnot{\hspace{-0,34cm}\nu_{h,f}} \grmptb \gnl
\glmptb \gnot{\hspace{-0,34cm}\nu_{g,f}} \grmptb \gcn{1}{1}{1}{2} \gnl
\gob{1}{\F(f)} \gob{2}{\F(g)} \gob{1}{\F(h)}
\gend
=
\gbeg{4}{5}
\gvac{1} \got{0}{\F(h)} \got{3}{\F(g)} \got{1}{\hspace{-0,34cm}\F(f)} \gnl
\gvac{1} \gcn{1}{1}{0}{1} \glmptb \gnot{\hspace{-0,34cm}\nu_{g,f}} \grmptb \gnl
\gvac{1} \glmptb \gnot{\hspace{-0,34cm}\nu_{h,f}} \grmptb \gcl{1} \gnl
\gvac{1} \gcn{1}{1}{1}{0} \glmptb \gnot{\hspace{-0,34cm}\nu_{h,g}} \grmptb \gnl
\gvac{1} \gob{0}{\F(f)} \gob{3}{\F(g)} \gob{1}{\hspace{-0,34cm}\F(h)}
\gend
\end{equation}
for all 1-endocells $f,g,h$, 
and the following unity-counity law 
\begin{equation}\eqlabel{YB-unity}
\gbeg{2}{5}
\got{1}{} \got{1}{\F(f)} \gnl
\gu{1} \gcl{1} \gnl
\glmptb \gnot{\hspace{-0,34cm}\nu_{1,f}} \grmptb \gnl
\gcl{1} \gcu{1} \gnl
\gob{1}{\F(f)} 
\gend
=\Id_{\F(f)}=
\gbeg{2}{5}
\got{1}{\F(f)}  \gnl
\gcl{1} \gu{1} \gnl
\glmptb \gnot{\hspace{-0,34cm}\nu_{f,1}} \grmptb \gnl
\gcu{1} \gcl{1} \gnl
\gob{3}{\F(f).} 
\gend
\end{equation} 
We call a Yang-Baxter operator for the identity 2-functor $\Id:\K\to\K$ 
{\em a Yang-Baxter operator of $\K$}. 
We reserve the notation $c$ for a Yang-Baxter operator on the identity 2-functor on $\K$. 
\end{defn} 



For 2-categories $\K=Del(\C)$ where $\C$ is a braided monoidal category, 
a class of Yang-Baxter operators $c$ is given by the braiding(s) of $\C$.

\begin{defn} \delabel{bilax}
Assume that $\K$ possesses a Yang-Baxter operator $c$. A {\em bilax functor} is a pair $(\F, \nu):(\K,c)\to\K'$ 
where $\F:\K\to\K'$ is simultaneously 
a lax and a colax functor and $\nu$ is a Yang-Baxter operator of $\F$, 
meaning that apart from the rule \equref{YBE F} two additional groups of rules hold: 
left and right lax distributive laws 
\begin{equation}\eqlabel{lax d.l.}
\gbeg{3}{5}
\got{0}{\F(h)} \got{3}{\F(g)} \got{0}{\F(f)} \gnl
\gmu \gcn{1}{1}{2}{0} \gnl
\gvac{1} \hspace{-0,34cm} \glmptb \gnot{\hspace{-0,34cm}\nu_{hg,f}} \grmptb  \gnl
\gvac{1} \gcn{1}{1}{1}{-1} \gcl{1} \gnl
\gob{1}{\F(f)} \gvac{1} \gob{1}{\F(hg)}
\gend
=
\gbeg{4}{5}
\gvac{1} \got{0}{\F(h)} \gvac{1} \got{1}{\F(g)} \got{2}{\F(f)} \gnl
\gvac{1} \gcn{1}{1}{0}{1} \glmptb \gnot{\hspace{-0,34cm}\nu_{g,f}} \grmptb \gnl
\gvac{1} \glmptb \gnot{\hspace{-0,34cm}\nu_{h,f}} \grmptb \gcl{1} \gnl
\gvac{1} \gcn{1}{1}{1}{-1} \gmu \gnl
\gob{1}{\F(f)} \gvac{1} \gob{2}{\F(hg)}
\gend, 
\quad
\gbeg{2}{5}
\got{1}{}\got{1}{\F(f)} \gnl
\gu{1} \gcl{1} \gnl
\glmptb \gnot{\hspace{-0,34cm}\nu_{1,f}} \grmptb \gnl
\gcl{1} \gcl{1} \gnl
\gob{0}{\F(f)} \gob{3}{\F(id)}
\gend=
\gbeg{2}{5}
\got{1}{\F(f)} \gnl
\gcl{1} \gu{1} \gnl
\gcl{2}  \gcl{2} \gnl
\gob{0}{\F(f)} \gob{3}{\F(id)}
\gend;
\qquad
\gbeg{3}{5}
\got{-1}{\F(h)} \got{4}{\F(g)} \got{-1}{\F(f)} \gnl
\gcn{1}{1}{0}{2} \gmu \gnl
\gvac{1} \hspace{-0,22cm} \glmptb \gnot{\hspace{-0,34cm}\nu_{h,gf}} \grmptb  \gnl
\gvac{1} \gcn{1}{1}{1}{-1} \gcl{1} \gnl
\gob{1}{\F(gf)} \gvac{1} \gob{1}{\F(h)}
\gend
=
\gbeg{5}{5}
\gvac{1} \got{0}{\F(h)} \gvac{1} \got{1}{\F(g)} \got{2}{\F(f)} \gnl
\gvac{1} \glmptb \gnot{\hspace{-0,34cm}\nu_{h,g}} \grmptb \gcn{1}{1}{2}{1} \gnl
\gvac{1} \gcl{1} \glmptb \gnot{\hspace{-0,34cm}\nu_{h,f}} \grmptb \gnl
\gvac{1} \gmu \gcl{1} \gnl
\gvac{1} \gob{1}{\F(gf)} \gob{3}{\F(h)}
\gend, 
\quad
\gbeg{2}{5}
\got{1}{\F(f)} \gnl
\gcl{1} \gu{1} \gnl
\glmptb \gnot{\hspace{-0,34cm}\nu_{f,1}} \grmptb \gnl
\gcl{1} \gcl{1} \gnl
\gob{0}{\F(id)} \gob{3}{\F(f)} 
\gend=
\gbeg{2}{5}
\got{3}{\F(f)} \gnl
\gu{1} \gcl{1} \gnl
\gcl{2}  \gcl{2} \gnl
\gob{0}{\F(id)} \gob{3}{\F(f)} 
\gend
\end{equation}
and left and right colax distributive laws 
\begin{equation}\eqlabel{colax d.l.}
\gbeg{3}{5}
\got{-1}{\F(f)} \got{5}{\F(hg)} \gnl
\gcn{1}{1}{0}{2} \gcn{1}{1}{2}{2} \gnl
\gvac{1} \hspace{-0,34cm} \glmpt \gnot{\hspace{-0,34cm}\nu_{f,hg}} \grmptb \gnl
\gvac{1} \hspace{-0,2cm} \gcmu \gcn{1}{1}{0}{1} \gnl
\gvac{1} \gob{0}{\F(h)} \gvac{1} \gob{1}{\F(g)} \gob{2}{\F(f)}
\gend
=
\gbeg{4}{5}
\gvac{1} \got{0}{\F(f)}\got{4}{\F(hg)} \gnl
\gvac{1} \gcl{1} \gcmu \gnl
\gvac{1} \glmptb \gnot{\hspace{-0,34cm}\nu_{f,h}} \grmptb \gcl{1} \gnl
\gvac{1} \gcn{1}{1}{1}{-1} \glmptb \gnot{\hspace{-0,34cm}\nu_{f,g}} \grmptb \gnl
\gvac{1} \gob{-1}{\F(h)} \gvac{2} \gob{0}{\F(g)} \gob{3}{\F(f)}
\gend,
\quad
\gbeg{2}{5}
\got{0}{\F(f)}\got{3}{\F(id)} \gnl
\gcl{1} \gcl{1} \gnl
\glmptb \gnot{\hspace{-0,34cm}\nu_{f,1}} \grmptb \gnl
\gcu{1} \gcl{1} \gnl
\gob{3}{\F(f)}
\gend=
\gbeg{2}{5}
\got{0}{\F(f)}\got{3}{\F(id)} \gnl
\gcl{1} \gcl{1} \gnl
\gcl{2}  \gcu{1} \gnl
\gob{1}{\F(f)}
\gend;
\qquad
\gbeg{2}{5}
\got{-1}{\F(gf)} \got{6}{\F(h)} \gnl
\gcn{1}{1}{0}{0} \gcn{1}{1}{2}{0} \gnl
\hspace{-0,34cm} \glmpt \gnot{\hspace{-0,34cm}\nu_{gf,h}} \grmptb \gnl
\gcn{1}{1}{1}{-1} \hspace{-0,2cm} \gcmu \gnl
\gob{0}{\F(h)} \gvac{1} \gob{1}{\F(g)} \gob{2}{\F(f)}
\gend
=
\gbeg{5}{5}
\got{3}{\F(gf)}\got{1}{\F(h)} \gnl
\gvac{1} \gcmu \gcl{1} \gnl
\gvac{1} \gcl{1} \glmptb \gnot{\hspace{-0,34cm}\nu_{f,h}} \grmptb \gnl
\gvac{1} \glmptb \gnot{\hspace{-0,34cm}\nu_{g,h}} \grmptb \gcn{1}{1}{1}{3} \gnl
\gvac{1} \gob{0}{\F(h)} \gob{3}{\F(g)} \gob{1}{\F(f)}
\gend,
\quad
\gbeg{2}{5}
\got{0}{\F(id)}\got{3}{\F(f)} \gnl
\gcl{1} \gcl{1} \gnl
\glmptb \gnot{\hspace{-0,34cm}\nu_{1,f}} \grmptb \gnl
\gcl{1} \gcu{1} \gnl
\gob{1}{\F(f)}
\gend=
\gbeg{2}{5}
\got{0}{\F(id)}\got{3}{\F(f)} \gnl
\gcl{1} \gcl{1} \gnl
\gcu{1}  \gcl{2} \gnl
\gob{3}{\F(f)}
\gend,
\end{equation}
and additionally the {\em bilaxity condition}
\begin{equation} \eqlabel{bilax} 
\gbeg{4}{5}
\got{2}{\F(gf)} \got{2}{\F(hk)} \gnl
\gcmu \gcmu \gnl
\gcl{1} \glmptb \gnot{\hspace{-0,34cm}\nu_{f,h}} \grmptb \gcl{1} \gnl
\gmu \gmu \gnl
\gob{2}{\F(gh)} \gob{2}{\F(fk)} \gnl
\gend
=
\gbeg{4}{5}
\got{1}{\F(gf)} \got{3}{\F(hk)} \gnl
\gwmu{3} \gnl
\glmp \gcmptb \gnot{\hspace{-0,78cm}\F(1c1)} \grmp \gnl
\gwcm{3} \gnl
\gob{1}{\F(gh)}\gvac{1}\gob{1}{\F(fk)}
\gend,
\qquad
\gbeg{3}{3}
\gu{1}  \gu{1} \gnl
\gcn{1}{1}{1}{-1} \gcl{1} \gnl
\gob{-1}{\F(id_A)} \gob{5}{\F(id_A)}
\gend=
\gbeg{3}{3}
\gvac{1} \gu{1} \gnl
\gwcm{3} \gnl
\gob{1}{\F(id_A)} \gob{3}{\F(id_A)}
\gend,
\qquad
\gbeg{2}{3}
\got{-1}{\F(id_A)} \got{5}{\F(id_A)} \gnl
\gcn{1}{1}{-1}{1} \gcl{1} \gnl
\gcu{1}  \gcu{1} \gnl
\gend=
\gbeg{3}{3}
\got{1}{\F(id_A)} \got{3}{\F(id_A)} \gnl
\gwmu{3} \gnl
\gvac{1} \gcu{1} \gnl
\gend,
\quad
\gbeg{1}{2}
\gu{1} \gnl
\gcu{1} \gnl
\gob{1}{}
\gend=
\Id_{id_{A}}
\end{equation}
holds for 1-cells $A\stackrel{k}{\to}B\stackrel{h}{\to}B\stackrel{f}{\to}B \stackrel{g}{\to}C$. 
\end{defn}

Observe that the unit laws in \equref{lax d.l.} (or the counit laws in \equref{colax d.l.}) together with the fourth rule in \equref{bilax} 
imply \equref{YB-unity}. We record that the unit laws of \equref{lax d.l.} (or the counit laws of \equref{colax d.l.}) imply 
\begin{equation}\eqlabel{eps-e}
\gbeg{2}{5}
\got{1}{} \got{1}{\F(id)} \gnl
\gu{1} \gcl{1} \gnl
\glmptb \gnot{\hspace{-0,34cm}\nu_{1,1}} \grmptb \gnl
\gcu{1} \gcl{1} \gnl
\gob{1}{} \gob{1}{\F(id)} 
\gend
=
\gbeg{1}{4}
\got{1}{\F(id)} \gnl
\gcu{1} \gnl
\gu{1} \gnl
\gob{1}{\F(id)} 
\gend
=
\gbeg{2}{5}
\got{1}{\F(id)}  \gnl
\gcl{1} \gu{1} \gnl
\glmptb \gnot{\hspace{-0,34cm}\nu_{1,1}} \grmptb \gnl
\gcl{1} \gcu{1} \gnl
\gob{1}{\F(id).} 
\gend
\end{equation}

\medskip

We briefly comment the term ``distributive law'' in equtions \equref{lax d.l.} and \equref{colax d.l.} (we also used it in \equref{phi-colax}). 
This term in the context of lax functors appeared in \cite[Definition 3.1]{FMS}, as the authors say: ``by
analogy to the distributive laws of monads, which have similar axioms''. On the other hand, monads and lax monoidal functors 
on (monoidal) categories can both be 
interpreted as monoids: the former are monoids in endofunctor categories, while the latter are monoids under Day convolution.
In accordance with this suggestive similarity we use the term ``distributive law'' also for pseudonatural transformations on lax/colax/bilax functors.  

\medskip

Let $(\K,c)$ and $(\K',d)$ be 2-categories with their respective Yang-Baxter operators, and assume that a bilax functor 
$(\F,\nu)$ is given acting between them.  If $\nu_{f,g}=d_{\F(f), \F(g)}$ for all composable 1-endocells $f,g$ in $\K$, we will 
say that $\F$ is a bilax functor with a {\em compatible Yang-Baxter operator}, and we will write simply $\F: (\K,c)\to (\K',d)$. 
In the case that the relation between $\nu$ and $d$ is not known, we will write $(\F,\nu): (\K,c)\to \K'$, as in \deref{bilax}. 

Given a bilax functor $(\F, \nu)$ the functors on endo-hom-categories
\begin{equation} \eqlabel{end-functors}
\F_{A,A}:  \K(A,A)\to\K'(\F(A), \F(A))
\end{equation}
are pre-bimonoidal in the sense of \cite[Section 2]{CS}, 
where we rely on the strictification theorem for monoidal categories. If instead of the Yang-Baxter operators one works 
with braidings on the endo-hom categories, then the functors $\F_{A,A}$ are bilax as in \cite[Section 19.9.1]{Agui}. 
The latter inspired the terminology in \deref{bilax}. 
While the analogues of the coherence conditions \equref{lax d.l.} and \equref{colax d.l.} do not appear explicitly in \cite[Section 2]{CS}, we incorporate them in accordance with the discussion in \cite[Section 19.9.1]{Agui} in our definition. Accordingly, 
we recover a ``braided bialgebra'' from \cite[Definition 5.1]{Tak}, as we will show further below. In the situation 
$\F: (\K,c)\to (\K',d)$ ({\em i.e.} that $\nu$ is compatible with $d$), the functors $\F_{A,A}$ are bimonoidal in the sense of 
\cite[Section 2]{CS}. The other way around, we clearly have:


\begin{ex} \exlabel{braided}
Let $(\C,\Phi_\C)$ and $(\D,\Phi_\D)$ be braided monoidal categories. We identify their braidings with Yang-Baxter operators $c$ and $d$ 
on their delooping categories $Del(\C)$ and $Del(\D)$, respectively. Any bilax functor $\F:(Del(\C),c)\to(Del(\C),d)$ with a compatible Yang-Baxter operator is a bimonoidal functor in the sense of \cite[Section 2]{CS}. 
\end{ex}


\begin{ex} \exlabel{triv-braided fun}
Let $(\C,\Phi)$ be a braided monoidal category and $1$ the trivial 2-category (it has a single 0-cell and only 
identity higher cells). Any bilax functor $\F:1\to Del(\C)$ with invertible $\nu$ 
which coincided with $\Phi$ can be identified with a bialgebra in $\C$. To see this, note that $\F$ determines and is determined by a 1-cell in $Del(\C)$, i.e. an object $B$ of $\C$, and a (co-)multiplication and (co-)unit on that 1-cell, which are subject to (co-)associativity and (co-)unitality due to $\F$ being lax and colax. On the other hand, $c$ of the trivial 2-category is trivial: 
it is the identity 2-cell on the identity 1-cell on $*$. Then the equations \equref{bilax} recover bialgebra axioms on $\F(id_*)$. (The rest of axioms of the bilax functor are automatically fulfilled by the braiding, 
and do not contribute any additional information.) 

If moreover $\F$ is a pseudofunctor, then it is trivial: the obtained bialgebra is isomorphic to the monoidal unit of $\C$. 
This can be proved considering the lax unitality structure of $\F$. 
\end{ex}




The following result is straightforwardly proved, see also \cite[Proposition 3.9]{CS}. 

\begin{prop}
Let $\F: (\K,c) \to (\K', d)$ and $\G: (\K', d) \to (\K'', e)$ be compatible bilax functors. 
  Then $\G\F: (\K, c) \to (\K'', e)$ is a compatible bilax functor with a Yang-Baxter operator $\nu_{g,f}:=\nu^\G_{\F(g),\F(f)}$, 
	where $\nu^\G$ is a Yang-Baxter operator of $\G$ and $g,f$ are 1-endocells of $\K$. 
\end{prop}

\subsection{Bimonads}

 The theory of monads and comonads in the context of 2-categories was introduced by Street in \cite{St1}. Recall that a monad in $\K$ , is 
a 1-endo-cell $t\in \K(A,A)$ endowed with 2-cells $\mu: tt \to t$ and $\eta: id_A \to t$ subject to associativity and unitality conditions. 
Dually, that is, swapping the direction of the structure 2-cells, one obtains the notion of a comonad in $\K$: a 1-endo-cell endowed with a coassociative and counital comultiplication. In order to differentiate notations for (co)lax functors and (co)monad structures,  
we will represent 
the multiplication and unit of a monad as well as the comultiplication and counit of a comonad by:
$$
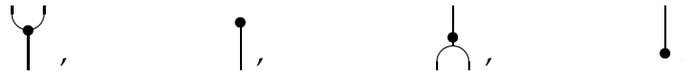
\gbeg{2}{3}
\gmuf{1} \gnl
\gcn{1}{1}{2}{2} \gnl
\gend,
\hspace{2cm}
\gbeg{1}{3}
\guf{1} \gnl
\gcl{1} \gnl
\gend,
\hspace{2cm}
\gbeg{2}{3}
\gcn{1}{1}{2}{2} \gnl
\gcmuf{1} \gnl
\gend,
\hspace{2cm}
\gbeg{1}{3}
\gcl{1} \gnl
\gcuf{1}  \gnl
\gend.
$$
One shows in a straightforward manner that lax functors preserve monads and colax functors preserve comonads. 
Specifically, for a monad $t$ in $\K$ and a lax functor $\F:\K\to\K'$, and a comonad $d$ in $\K$ and a colax functor $\G$
we have the following monad and comonad structures:
\begin{equation} \eqlabel{dot-structures}
\gbeg{3}{4}
\got{1}{\F(t)} \gvac{1} \got{1}{\F(t)} \gnl
\gwmuf{3} \gnl
\gvac{1} \gcl{1} \gnl
\gob{3}{\F(t)}
\gend
=
\gbeg{4}{5}
\got{1}{\F(t)} \gvac{1} \got{1}{\F(t)} \gnl
\gwmu{3} \gnl
\glmp \gnot{\F(\nabla)} \gcmptb \grmp \gnl
\gvac{1} \gcl{1} \gnl
\gob{3}{\F(t)}
\gend,
\hspace{0,6cm}
\gbeg{1}{3}
\guf{1} \gnl
\gcl{1} \gnl
\gob{1}{\F(t)} \gnl
\gend =
\gbeg{4}{3}
\gvac{1} \gu{1} \gnl
\glmp \gnot{\F(\eta)} \gcmptb \grmp \gnl
\gvac{1}\gob{1}{\F(t)}
\gend,
\hspace{0,5cm}
\gbeg{3}{4}
\got{3}{\F(d)} \gnl
\gvac{1} \gcl{1} \gnl
\gwcmf{3} \gnl
\gob{1}{\F(d)}\gvac{1}\gob{1}{\F(d)}
\gend
=
\gbeg{4}{5}
\got{3}{\F(d)} \gnl
\gvac{1} \gcl{1} \gnl
\glmp \gnot{\F(\Delta)} \gcmptb \grmp \gnl
\gwcm{3} \gnl
\gob{1}{\F(d)}\gvac{1}\gob{1}{\F(d)}
\gend,
\hspace{0,5cm}
\gbeg{1}{3}
\got{1}{\F(d)} \gnl
\gcl{1} \gnl
\gcuf{1} \gnl
\gend =
\gbeg{4}{3}
\gvac{1}\got{1}{\F(d)} \gnl
\glmp \gnot{\F(\Epsilon)} \gcmptb \grmp \gnl
\gvac{1} \gcu{1} \gnl
\gend.
\end{equation}

We are going to introduce bimonads in 2-categories with respect to Yang-Baxter operators. 
Observe that their 1-categorical analogue is different than the bimonads of \cite{Wisb} and \cite{Moe} in ordinary categories, 
but they are a particular instance of $\tau$-bimonads and bimonads in 2-categories from \cite{F2}. Whereas bimonads in \cite{Moe} 
are opmonoidal monads on monoidal categories, bimonads in \cite{Wisb} are monads and comonads on a not necessarily monoidal category 
with compatibility conditions that involve a distributive law $\lambda$. The bimonads in \cite[Definition 7.1]{F2} generalize the latter 
to 2-categories and are equipped with an analogous 2-cell {\em i.e.} distributive law $\lambda$. The $\tau$-bimonads that we also introduced in \cite{F2}, are a particular case of bimonads, where the 2-cell $\lambda$ is given in terms of a 2-cell $\tau$ which is a distributive law both on the left and on the right, both with respect to monads and comonads. 
In the special case of a Yang-Baxter operator coming from a braiding one gets examples \`a la \cite{Wisb}.


\begin{defn} \delabel{c-bimnd}
Let $(\K,c)$ be a 2-category with a Yang-Baxter operator $c$ and suppose that $b\in\K$ is a 1-endo-cell endowed with a monad and comonad structure. 
We call $b$ a {\em $c$-bimonad} if $c_{b,b}$ is a distributive law both on the left and on the right, both with respect to monads and comonads, 
and the following compatibilities hold:
\begin{equation} \eqlabel{c-bimonad} 
\gbeg{4}{7}
\got{2}{b} \got{2}{b} \gnl
\gcn{1}{1}{2}{2} \gvac{1} \gcn{1}{1}{2}{2} \gnl
\gcmuf \gnl \gvac{2} \gcmuf \gnl \gnl%
\gcl{1} \glmptb \gnot{\hspace{-0,34cm}c_{b,b}} \grmptb \gcl{1} \gnl
\gmuf \gnl \gmuf \gnl \gnl
\gcn{1}{1}{2}{2} \gvac{1}\gcn{1}{1}{2}{2} \gnl
\gob{2}{b} \gob{2}{b} \gnl
\gend=
\gbeg{3}{5}
\got{1}{b} \got{3}{b} \gnl
\gwmuf{3} \gnl
\gvac{1} \gcl{1} \gnl
\gwcmf{3} \gnl
\gob{1}{b}\gvac{1}\gob{1}{b}
\gend,\quad
\gbeg{2}{3}
\got{1}{b} \got{1}{b} \gnl
\gcl{1} \gcl{1} \gnl
\gcuf{1}  \gcuf{1} \gnl
\gend=
\gbeg{2}{3}
\got{1}{b} \got{1}{b} \gnl
\gmuf \gnl \gnl
\gvac{1} \hspace{-0,2cm} \gcuf{1} \gnl
\gob{1}{}
\gend, \quad
\gbeg{2}{3}
\guf{1}  \guf{1} \gnl
\gcl{1} \gcl{1} \gnl
\gob{1}{b} \gob{1}{b}
\gend=
\gbeg{2}{4}
\guf{1} \gnl
\gcl{1} \gnl
\hspace{-0,34cm} \gcmuf \gnl \gnl
\gob{1}{b} \gob{1}{b}
\gend\hspace{-0.3cm},\quad
\gbeg{1}{2}
\guf{1} \gnl
\gcuf{1} \gnl
\gob{1}{}
\gend=
\Id_{\Id_{A}}.
\end{equation}
\end{defn}

  The following observation is inspired by Benabo\'u:

\begin{lma} \lelabel{blx from triv}
There exists a bijection between $c$-bimonads in $(\K,c)$ and compatible bilax functors $1 \to (\K,c)$. 
It is given by mapping any $c$-bimonad $b: A\to A$ in $(\K,c)$ to the bilax functor $\T : 1 \to \K$  with $\T(id_{*})= b$, 
whose bilax structure 2-cells agree with the respective 2-cells of the $c$-bimonad $b$. 
\end{lma}


\begin{proof}
As observed by Benabo\'u, a lax functor from $1$ to $\K$ defines and is defined by a monad in $\K$. 
Likewise, the colaxity of $\T : 1 \to \K$ is equivalent to $\T(id_*)$ being a comonad. The equations	\equref{lax d.l.} and 
\equref{colax d.l.} correspond to the four distributive laws of $c_{b,b}$ with $b=\T(id_*)$ and the bilaxity conditions given in \equref{bilax} 
correspond to conditions \equref{c-bimonad}. 
Such a bilax functor is also a $\tau$-bimonad in the sense of \cite[Definition 7.1]{F2} (it is a monad and a comonad with 
a monad-morphic and a comonad-morphic distributive law on both sides, in the sense of \equref{lax d.l.} and \equref{colax d.l.}). 
\qed\end{proof}

Similarly, we clearly have:

\begin{lma} \lelabel{T(id)}
Any bilax functor $(\T,\nu) : 1 \to \K$ determines a $\nu$-bimonad $\T(id_{*})= b$. 
\end{lma}

\begin{prop} \prlabel{preserve bimnd}
Let $(\F,\nu):(\K,c)\to\K'$ be a bilax functor and $b\in\K$ a $c$-bimonad. Then $\F(b)$ is a $\nu$-bimonad in $\K'$. 
That is, it satisfies the last three axioms of \equref{c-bimonad} and a variant of the first axiom shown below:
\begin{equation} \eqlabel{tau-bimonad}
\gbeg{4}{7}
\got{2}{\F(b)} \got{2}{\F(b)} \gnl
\gcn{1}{1}{2}{2} \gvac{1} \gcn{1}{1}{2}{2} \gnl
\gcmuf \gnl \gvac{2} \gcmuf \gnl \gnl%
\gcl{1} \glmptb \gnot{\hspace{-0,34cm}\nu_{b,b}} \grmptb \gcl{1} \gnl
\gmuf \gnl \gmuf \gnl \gnl
\gcn{1}{1}{2}{2} \gvac{1}\gcn{1}{1}{2}{2} \gnl
\gob{2}{\F(b)} \gob{2}{\F(b)} \gnl
\gend=
\gbeg{3}{5}
\got{1}{\F(b)} \got{3}{\F(b)} \gnl
\gwmuf{3} \gnl
\gvac{1} \gcl{1} \gnl
\gwcmf{3} \gnl
\gob{1}{\F(b)}\gvac{1}\gob{1}{\F(b).}
\gend
\end{equation}
Moreover, $\F(b)$ is a $\tau$-bimonad in the sense of \cite{F2}. 
\end{prop}

\begin{proof}
The claim 
follows by the naturality of the (co)lax structure of $\F$ with respect to the 2-cells from the (co)monad structures of $b$. 
(For the three compatibilities of the (co)unit structures for $\F(b)$, alternatively, apply $\F(\eta_b)$ and $\F(\Epsilon_b)$ at suitable places in the last three equations in \equref{bilax} and use the naturality of the (co)lax structure of $\F$.) 
The last statement follows from \equref{lax d.l.} and \equref{colax d.l.}. 
\qed\end{proof}

\begin{ex} 
We have that $\F(id_A)$ for all $A\in\Ob\K$ are $\nu$-bimonads in $\K'$. 
\end{ex}

\medskip

Given that the notion of a Yang-Baxter operator is more general than that of 
a braiding, a $\nu$-bimonad generalizes the notion of a bialgebra. Moreover, we have that the notion of a bilax functor $\F:\K\to\K'$ 
recovers that of a ``braided bialgebra'' from \cite[Definition 5.1]{Tak}, in the case $\K=Del(\C)$, where $\C$ is a monoidal category with a single object. 

\begin{ex} \exlabel{my bimonads}
Let $(\K,c)$ and $(\K',d)$ be two 2-categories with their respective Yang-Baxter operators, and let $b:A_0\to A_0$ be a $d$-bimonad in $\K'$.
We define $\F_b(A)=A_0, \F_b(x)=b, \F_b(\alpha)=\Id_b$ for all objects $A$, 1-cells $x$ and 2-cells $\alpha$ in $\K$. 
Then $\F_b$ is a bilax functor with $\nu_{f,g}=d_{b,b}$ for all 1-endocells $f,g$ in $\K$. 
(Alternatively, instead of the Yang-Baxter operator $d$ in $\K'$ one could require 
a Yang-Baxter operator $\nu$ on $\F_b$, take a $\nu$-bimonad $b$ and obtain a bilax functor ($\F_b,\nu)$.) 
\end{ex}

The following claim is directly proved:

\begin{lma}
For a comonad $d:A\to A$ and a colax transformation between two colax functors $\chi:\F\Rightarrow\G$ it holds:
$$
\gbeg{3}{6}
\got{1}{\chi_A}\got{2}{\F(d)} \gnl
\gcl{1} \gcn{1}{1}{2}{2} \gnl
\gcl{1} \gcmuf \gnl \gnl
\glmptb \gnot{\hspace{-0,34cm}\chi_d} \grmptb \gcl{1} \gnl
\gcn{1}{1}{1}{-1} \glmptb \gnot{\hspace{-0,34cm}\chi_d} \grmptb \gnl
\gob{-1}{\G(d)} \gvac{2} \gob{1}{\G(d)} \gob{1}{\chi_A}
\gend=
\gbeg{3}{5}
\got{1}{\chi_A}\got{2}{\F(d)} \gnl
\gcn{1}{1}{2}{2} \gcn{1}{1}{2}{2} \gnl
\gvac{1} \hspace{-0,34cm} \glmpt \gnot{\hspace{-0,34cm}\chi_{d}} \grmptb \gnl
\gvac{1} \hspace{-0,2cm} \gcmuf \gnl \gvac{2} \gcn{1}{1}{0}{1} \gnl
\gvac{1} \gob{0}{\G(d)} \gvac{1} \gob{1}{\G(d)} \gob{1}{\chi_A}
\gend;
\qquad\quad
\gbeg{3}{5}
\got{1}{\chi_A}\got{2}{\F(d)}\gnl
\gcl{1} \gcl{1} \gnl
\glmptb \gnot{\hspace{-0,34cm}\chi_{d}} \grmptb \gnl
\gcuf{1} \gcl{1} \gnl
\gob{3}{\chi_A}
\gend=
\gbeg{3}{5}
\got{1}{\chi_A}\got{2}{\F(d)}\gnl
\gcl{1} \gcl{1} \gnl
\gcl{2}  \gcuf{1} \gnl \gnl
\gob{1}{\chi_A}
\gend.
$$
\end{lma}


Monads and comonads in $\K$ can act on other 1-cells of $\K$. That is, for example a left module over a bimonad $b:A\to A$ is a 1-cell 
$x:A'\to A$ endowed with a left action $\rhd : bx \Rightarrow x$ of $b$ (in \cite[Definition 2.3]{F1} the axioms are expressed in string diagrams). Note that in category theory, what we call a $b$-(co)module is also referred to as a $b$-(co)algebra, see for example \cite{LS}.

%


\begin{prop} \prlabel{(co)mod str}
Let $\F, \G$ be two bilax functors and $b: A\to A$ a $c$-bimonad in $\K$ (indeed, a monad and comonad
satisfying the two last identities in \equref{c-bimonad}).

For a colax transformation $\phi:\F\Rightarrow\G$ of colax functors, \equref{left G(b)-comod} defines a left $\G(b)$-comodule structure 
on $\phi(A)$.

Dually, for a lax transformation $\psi:\F\Rightarrow\G$ of lax functors, \equref{left G(b)-mod} defines a left $\G(b)$-module structure 
on $\psi(A)$.
\begin{center} \hspace{-0,6cm}
\begin{tabular}{p{6.6cm}p{1cm}p{6.6cm}}
\begin{equation} \eqlabel{left G(b)-comod}
\gbeg{2}{4}
\got{3}{\phi_A} \gnl
 \gvac{1} \gcl{1}\gnl
\glcm \gnl
\gob{0}{\G(b)} \gob{3}{\phi_A}
\gend
:=
\gbeg{3}{5}
\gvac{1} \got{1}{\phi_A} \gnl
\gvac{1} \gcl{1} \gu{1} \gnl
\gvac{1} \glmptb \gnot{\hspace{-0,34cm}\phi_{id_A}} \grmptb \gnl
\glmp \gnot{\hspace{-0,34cm}\G(\eta)} \grmptb\gcl{1} \gnl
\gvac{1} \gob{0}{G(b)} \gob{3}{\phi_A}
\gend
\stackrel{nat.}{=}
\gbeg{2}{5}
\got{1}{\phi_A} \gnl
\gcl{1} \guf{1} \gnl
\glmptb \gnot{\hspace{-0,34cm}\phi_b} \grmptb \gnl
\gcl{1} \gcl{1} \gnl
\gob{0}{G(b)} \gob{3}{\phi_A}
\gend
\end{equation} & & 
\begin{equation} \eqlabel{left G(b)-mod} 
\gbeg{2}{4}
\got{0}{\G(b)} \got{3}{\psi_A} \gnl
\glm \gnl
 \gvac{1} \gcl{1}\gnl
\gob{3}{\psi_A}
\gend
:=
\gbeg{3}{5}
\gvac{1} \got{0}{\G(b)} \got{3}{\psi_A} \gnl
\glmp \gnot{\hspace{-0,34cm}\G(\Epsilon)} \grmptb\gcl{1} \gnl
\gvac{1} \glmptb \gnot{\hspace{-0,34cm}\psi_{id_A}} \grmptb \gnl
\gvac{1} \gcl{1} \gcu{1} \gnl
\gob{3}{\psi_A}
\gend
\stackrel{nat.}{=}
\gbeg{2}{5}
\got{0}{\G(b)} \got{3}{\psi_A} \gnl
\gcl{1} \gcl{1} \gnl
\glmptb \gnot{\hspace{-0,34cm}\psi_b} \grmptb \gnl
\gcl{1} \gcuf{1} \gnl
\gob{1}{\psi_A}
\gend
\end{equation}
\end{tabular}
\end{center}
\end{prop}

\begin{proof}
We indicate the steps of the proof for the comodule structure. Starting from
$\gbeg{3}{3}
\gvac{1} \glcm \gnl
\gcn{2}{1}{3}{1}  \gcl{1} \gnl
\gcl{1} \glcm \gnl
\gend$ apply first the second rule in \equref{bilax}, then the first rule in \equref{phi-colax}, naturality of
$\gbeg{2}{1}
\gcmu \gnl
\gend$, naturality of $\chi$, third rule in \equref{c-bimonad} for $b$, and lastly naturality of $\chi$.

For the counitality, starting from
$\gbeg{2}{3}
\glcm \gnl
\gcuf{1} \gcl{1} \gnl
\gend$ apply first the fourth rule in \equref{c-bimonad}, the second rule in \equref{phi-colax}, and lastly the fourth
rule in \equref{bilax}.
\qed\end{proof}

For $\F=\G$ acting on the trivial 2-category $\K=1$, and $\phi$ and $\psi$ being endo-transformations, \prref{(co)mod str} recovers
\cite[Proposition 2.4]{F1}, proved for 2-(co)monads and their distributive laws. The latter is a 2-categorical formulation for a fact 
possibly used in (braided) monoidal categories by different authors, but we are not aware of an exact reference. 


It is important to note that modules/comodules over a $c$-bimonad in $\K$ do not form a monoidal category (unless the Yang-Baxter operator $c$ is a half-braiding).

\subsection{Module comonads, comodule monads and relative bimonad modules} \sslabel{mod}

In \cite{LS} the notion of a wreath was introduced as monad in the free completion 2-category $\EM^M(\K)$ of the 2-category $\Mnd(\K)$ 
of monads in $\K$ under the Eilenberg-Moore construction. Dually, cowreaths are comonads in $\EM^C(\K)$, where the latter is the analogous 
completion of the 2-category $\Comnd(\K)$ of comonads. In \cite{F1} the first author introduced the 2-category $\bEM(\K)$ 
so that there are forgetful 2-functors $\bEM(\K)\to\EM^M(\K)$ and $\bEM(\K)\to\EM^C(\K)$. Moreover, in {\em loc. cit.} biwreaths 
were introduced as bimonads in $\bEM(\K)$. Biwreaths as a notion integrate both 
wreaths and cowreaths as well as their mixed versions: mixed wreaths and mixed cowreaths. In particular, a biwreath also behaves like a 
``(co)module (co)monad'' with respect to monad-morphic or comonad-morphic distributive law in $\K$, where the highlighted notions in the 
2-categorical setting were introduced in \cite{F1}. 

In the present paper, similarly to the above-mentioned idea (see diagrams (67) and (65) of {\em loc. cit.}), but now with respect to Yang-Baxter operators, 
we will consider the notions that we introduce in the definition below. 

For the sake of examples that we will study further below, we record that in \cite{F1} the 2-category $\Bimnd(\K)$ of bimonads in $\K$ (with 
respect to distributive laws) was defined, so that there are inclusion and projection 2-functors $E_B: \Bimnd(\K)\to\bEM(\K)$ and 
$\pi: \bEM(\K)\to\Bimnd(\K)$ which are identities on 0- and 1-cells. In \cite{F2} we have considered a variation of the 2-category $\Bimnd(\K)$ 
by changing the definition of 2-cells. It is this other version of the 2-category that we will be interested in here. 
We will recall it in \exref{bimnd-K}. 

\begin{defn} \delabel{mod-comod}
Let $b:A\to A$ be a $c$-bimonad in a 2-category $(\K,c)$ with a Yang-Baxter operator. Let $d:A\to A$ be a comonad and a right $b$-module, 
and $t:A\to A$ a monad and a right $b$-module.  
We say that $d$ is a (right) {\em module comonad} if the left two equations below hold,  
and that $t$ is a (right) {\em comodule monad} if the right two equations below hold:
$$
\gbeg{4}{6}
\got{2}{d} \got{2}{b} \gnl
\gcn{1}{1}{2}{2} \gvac{1} \gcn{1}{1}{2}{2} \gnl
\gcmuf \gnl \gvac{2} \gcmuf \gnl \gnl%
\gcl{1} \glmptb \gnot{\hspace{-0,34cm}c_{d,b}} \grmptb \gcl{1} \gnl
\grm \grm \gnl
\gob{1}{d} \gob{3}{d} \gnl
\gend
=
\gbeg{3}{5}
\gvac{1} \got{1}{d} \got{1}{b} \gnl
\gvac{1} \grm \gnl
\gvac{1} \gcl{1} \gnl
\gwcmf{3} \gnl
\gob{1}{d}\gvac{1}\gob{1}{d}
\gend,
\quad
\gbeg{2}{3}
\got{1}{d} \got{1}{b} \gnl
\grm \gnl
\gcuf{1} \gnl
\gend
=
\gbeg{3}{3}
\got{1}{d} \got{1}{b} \gnl
\gcl{1} \gcl{1} \gnl
\gcuf{1}  \gcuf{1} \gnl
\gend;
\qquad
\gbeg{4}{6}
\got{1}{t} \got{3}{t} \gnl
\grcm \grcm \gnl
\gcl{1} \glmptb \gnot{\hspace{-0,34cm}c_{b,t}} \grmptb \gcl{1} \gnl
\gmuf \gnl \gmuf \gnl \gnl%
\gcn{1}{1}{2}{2} \gvac{1} \gcn{1}{1}{2}{2} \gnl
\gob{2}{t} \gob{2}{b} \gnl
\gend
=
\gbeg{3}{5}
\got{1}{t} \gvac{1} \got{1}{t} \gnl
\gwmuf{3} \gnl
\gvac{1} \gcl{1} \gnl
\gvac{1} \grcm \gnl
\gvac{1} \gob{1}{t} \gob{1}{b}
\gend,
\quad
\gbeg{2}{3}
\guf{1} \gnl
\grcm \gnl
\gob{1}{t} \gob{1}{b} \gnl
\gend
=
\gbeg{3}{3}
\guf{1}  \guf{1} \gnl
\gcl{1} \gcl{1} \gnl
\gob{1}{t} \gob{1}{b.} \gnl
\gend
$$
\end{defn}

The left hand-side versions of these notions can be clearly deduced. 

\medskip

We continue with some simple yoga of (co)lax and bilax functors.

Consider 1-cells:
$
\bfig
\putmorphism(20,53)(1,0)[``t]{300}1a
\putmorphism(0,30)(1,0)[A`A`]{340}0a
\putmorphism(20,13)(1,0)[``d]{300}1b

\putmorphism(370,30)(1,0)[`B`x]{300}1a

\efig
$  in $\K$, where $t$ is a monad, $d$ is a comonad and $x$ is a right $t$-module (via a 2-cell $\lhd$) and a right $d$ comodule 
(via a 2-cell $\rho$). Recall that, as we already used before, lax functors preserve monads and colax functors preserve comonads. 
Given a lax functor $\F$ and a colax functor $\G$ one has that $\F(x)$ is a right $\F(t)$-module and a right $\G(d)$-comodule 
with structure 2-cells:
\begin{equation} \eqlabel{induced (co)mods}
\gbeg{3}{4}
\got{1}{\F(x)} \gvac{1} \got{1}{\F(t)} \gnl
\gcl{1} \gcn{1}{1}{3}{1} \gnl
\grm \gnl
\gob{1}{\F(x)}
\gend
=
\gbeg{4}{4}
\got{1}{\F(x)} \gvac{1} \got{1}{\F(t)} \gnl
\gwmu{3} \gnl
\glmp \gnot{\F(\lhd)} \gcmptb \grmp \gnl
\gob{3}{\F(x)}
\gend,
\hspace{1,5cm}
\gbeg{3}{4}
\got{1}{\G(x)} \gnl
\grcm \gnl
\gcl{1} \gcn{1}{1}{1}{3} \gnl
\gob{1}{\G(x)}\gvac{1}\gob{1}{\G(d)}
\gend
=
\gbeg{4}{4}
\got{3}{\G(x)} \gnl
\glmp \gnot{\F(\rho)} \gcmptb \grmp \gnl
\gwcm{3} \gnl
\gob{1}{\G(x)}\gvac{1}\gob{1}{\G(d).}
\gend
\end{equation}
The analogous claims hold on left sides.

\begin{prop} \prlabel{bilax preserves comod mod}
Bilax functors $\F: (\K,c)\to (\K',d)$ with compatible Yang-Baxter operators preserve module comonads and comodule monads. 
\end{prop}


\begin{proof}
The arguments for showing that $\F$ preserves comodule monads are analogous to proving that it maps module comonads to module comonads.
Therefore, we will restrict ourselves to the latter. 
Let $b : A \to A$ be a bimonad in $(\K,c)$ and  $x : A \to A$ a $b$-module comonad.
  The bilaxity of $\F$ implies that $\F(b)$ is a $d$-bimonad with an action on the comonad  $\F(x)$. 
  Hence, we only have to show that the first two compatibility conditions in \deref{mod-comod} hold. 
  Using the functoriality of $\F$ we have: 
$$
\gbeg{4}{6}
\got{2}{\F(x)} \got{2}{\F(b)} \gnl
\gcn{1}{1}{2}{2} \gvac{1} \gcn{1}{1}{2}{2} \gnl
\gcmuf \gnl \gvac{2} \gcmuf \gnl \gnl%
\gcl{1} \glmptb \gnot{\hspace{-0,34cm}d_{x,b}} \grmptb \gcl{1} \gnl
\grm \grm \gnl
\gob{1}{\F(x)} \gob{3}{\F(x)} \gnl
\gend
=
\gbeg{4}{7}
\got{1}{\F(x)} \got{5}{\F(b)} \gnl
\glmptb \gnot{\hspace{-0,34cm}\F\w(\Delta_d)} \grmp \glmp \gnot{\hspace{-0,34cm}\F\w(\Delta_b)} \grmptb \gnl
\gwmu{4}  \gnl
\gvac{1} \hspace{-0,34cm} \glmp \gnot{\F(1c1)} \gcmptb \grmp \gnl
\gvac{1} \hspace{-0,22cm} \gwcm{4} \gnl
\gvac{1} \glmptb \gnot{\hspace{-0,38cm}\F(\lhd)} \grmp \glmp \gnot{\hspace{-0,38cm}\F(\lhd)} \grmptb \gnl
\gvac{1} \gob{1}{\F(x)} \gob{5}{\F(x)}
\gend
= 
\gbeg{3}{7}
\got{1}{\F(x)} \got{3}{\F(b)} \gnl
\gwmu{3}  \gnl
\glmp \gcmptb \gnot{\hspace{-0,74cm}\F(\Delta_d\Delta_b)} \grmp \gnl
\glmp \gnot{\F(1c1)} \gcmptb \grmp \gnl
\glmp \gcmptb\gnot{\hspace{-0,78cm}\F(\lhd\lhd)} \grmp \gnl
\gwcm{3} \gnl
\gob{1}{\F(x)} \gob{3}{\F(x)}
\gend
=
\gbeg{3}{6}
\got{1}{\F(x)} \got{3}{\F(b)} \gnl
\gwmu{3}  \gnl
\glmp \gcmptb \gnot{\hspace{-0,74cm}\F(\lhd)} \grmp \gnl
\glmp \gcmptb \gnot{\hspace{-0,74cm}\F(\Delta_d)} \grmp \gnl
\gwcm{3} \gnl
\gob{1}{\F(x)} \gob{3}{\F(x)}
\gend
=
\gbeg{3}{5}
\gvac{1} \got{1}{\F(x)} \got{3}{\F(b)} \gnl
\gvac{1} \gcl{1} \gcn{1}{1}{3}{1} \gnl
\gvac{1} \grm \gnl
\gwcmf{3} \gnl
\gob{1}{\F(x)} \gob{3}{\F(x).}
\gend
$$
The compatibility of the action of $\F(b)$ with the counit of $\F(x)$ is an immediate consequence of $\F$ being lax. 
\qed\end{proof}

\medskip

For the next property we introduce the following notion:

\begin{defn} \delabel{rel Hopf}
Let $b:A\to A$ be a $c$-bimonad in a 2-category $(\K,c)$ with a Yang-Baxter operator and let $t:A\to A$ be a right $b$-comodule monad.
A right $t$-module and a right $b$-comodule $x:A\to B$ is a right {\em relative $t\x b$-module} if the following relation holds:
$$
\gbeg{4}{5}
\got{1}{x} \got{3}{t} \gnl
\grcm \grcm \gnl%
\gcl{1} \glmptb \gnot{\hspace{-0,34cm}c_{b,t}} \grmptb \gcl{1} \gnl
\grm \gmuf \gnl \gnl
\gob{1}{x} \gob{4}{b} \gnl
\gend
=
\gbeg{3}{6}
\got{1}{x} \got{1}{t} \gnl
\gcl{1} \gcl{1} \gnl
\grm \gnl
\grcm \gnl
\gcl{1} \gcl{1} \gnl
\gob{1}{x} \gob{1}{b.}
\gend
$$
Morphisms of right relative $t\x b$-modules are right $t$-linear and right $b$-colinear 2-cells in $\K$. 
Analogously, one defines a left {\em relative $t\x b$-module}. If $t=b$ we call $x$ a {\em relative bimonad module}.
\end{defn}

The above notions correspond to those of relative Hopf modules \cite{CF} and Hopf modules \cite{Ly} in braided monoidal categories, 
which in turn are categorifications of Hopf modules introduced in \cite{LaSw}. 
Obviously, $b$ itself is a relative bimonad module.

\begin{ex} \exlabel{relative mod}
Let $\F: (\K,c)\to (\K',d)$ be a bilax functor with compatible Yang-Baxter operator. 
For any 1-cell $x:A\to B$ the 1-cell $\F(x)$ is a left relative bimonad module over $\F(id_B)$ and a right relative bimonad module over $\F(id_A)$.
This follows from the first equation in \equref{bilax} by setting $id_B,id_B, id_B, x$, respectively $x, id_A, id_A, id_A$ for the 1-cells
$g,f,h,k$.
\end{ex}

Analogously to \prref{bilax preserves comod mod} we get the following result:

\begin{prop} \prlabel{relative mod}
Bilax functors $\F: (\K,c)\to (\K',d)$ with compatible Yang-Baxter operators preserve relative bimonad modules. 
\end{prop}

\medskip

Hopf bimodules, for Hopf algebras over a field, appeared in the construction of bicovariant differential calculi over a Hopf algebra in 
\cite{Wor}. They were generalized in \cite[Section 4.2]{Besp} to the context of a braided monoidal category $\C$. For a bialgebra $B$ in $\C$ 
a Hopf bimodule is a $B$-bimodule $M$ in $\C$ which is moreover a $B$-bicomodule in the monoidal category of $B$-bimodules ${}_B\C_B$ 
(for the structures on $B$ itself the regular (co)action on $B$ and the diagonal action on tensor product of comodules are used). 
The latter means 
that both left and right $B$-comodule structures of $M$ are left and right $B$-bimodule morphisms, meaning that there are four conditions to be 
fulfilled. Together with simultaneously $B$-linear and $B$-bicolinear morphisms Hopf bimodules make a category denoted by ${}^B_B\C_B^B$. 
We mark that the name ``Hopf bimodules'' is somewhat misleading, as the Hopf structure on the bialgebra $B$ is not 
necessary here. 

Substituting a braided  monoidal category $\C$ and a bialgebra $B$ in it with a monoidal category with a Yang-Baxter operator $c$ and a 
$c$-bimonad in it, we can consider the analogous category of Hopf bimodules ${}^B_B(\C,c)_B^B$, where $\C$ and $B$ now have these new meanings. 

\begin{cor}
Let $\F:(\K,c)\to(\K',d)$ be a bilax functor with compatible Yang-Baxter operator. 
The functors \equref{end-functors} for every $A\in\Ob\K$ factor through the category of Hopf bimodules over the $d$-bimonads $\F(id_A)$ 
in $(\K'(\F(A), \F(A)),d)$.
\end{cor}

\begin{proof}
For any 1-endocell $x:A\to A$ we should check that $\F(x)$ satisfies the four relations. Two of them are satisfied by \exref{relative mod}, 
which mean that the left coaction is left linear and that right coaction is right linear. The other two, meaning the two mixed versions of compatibilities, 
one gets by setting $x=f$, respectively $x=h$, and the resting three 1-cells to be identities, in the first equation in \equref{bilax}. 

To check the claim on morphisms, observe that for any 2-cell $\alpha:x\to y$ in $\K$, {\em i.e.} morphism in $\K(A,A)$,
$\F(\alpha)$ is $\F(id_A)$-(co)linear by the naturality of the (co)lax structure of $\F$.
\qed\end{proof}

\medskip

We record here some direct consequences for a bilax functor $(\F,\nu)$ from the first equation in \equref{bilax}.  
For simplicity, we may consider $\F:(\K,c)\to(\K',d)$ to have a compatible Yang-Baxter operator. 
By \equref{induced (co)mods} note that $\F(x)$ is a bi(co)module over $\F(id)$ for any 1-cell $x$ (acting among suitable 0-cells). 
Then we may write: 
\begin{center} \hspace{-1,4cm}
\begin{tabular} {p{6cm}p{1cm}p{6cm}} 
\begin{equation} \eqlabel{mod coalg-l}
\gbeg{4}{5}
\got{1}{\F(id)} \got{3}{\F(fk)} \gnl
\gcmu \gcmu \gnl
\gcl{1} \glmptb \gnot{\hspace{-0,34cm}\nu_{1,f}} \grmptb \gcl{1} \gnl
\glm \glm \gnl
\gvac{1} \gob{1}{\F(f)} \gob{3}{\F(k)}
\gend=
\gbeg{3}{5}
\got{1}{\F(id)} \got{3}{\F(fk)} \gnl
\gcn{1}{1}{1}{3} \gvac{1} \gcl{1} \gnl
\gvac{1}\glm \gnl
\gvac{1} \gwcm{3} \gnl
\gvac{1} \gob{1}{\F(f)} \gob{3}{\F(k)}
\gend
\end{equation} & &
\begin{equation}\eqlabel{comod alg-l}
\gbeg{4}{5}
\got{3}{\F(f)} \got{1}{\F(k)} \gnl
\glcm \glcm \gnl
\gcl{1} \glmptb \gnot{\hspace{-0,34cm}\nu_{f,1}} \grmptb \gcl{1} \gnl
\gmu \gmu \gnl
\gob{2}{\F(id)} \gob{2}{\F(fk)}
\gend=
\gbeg{2}{5}
\gvac{1} \got{1}{\F(f)} \got{3}{\F(k)} \gnl
\gvac{1} \gwmu{3} \gnl
\gvac{1} \glcm \gnl
\gcn{1}{1}{3}{1} \gvac{1} \gcl{1} \gnl
\gob{1}{\F(id)} \gob{3}{\F(fk)}
\gend
\end{equation}
\end{tabular}
\end{center}

\begin{center} \hspace{-1,4cm}
\begin{tabular} {p{6cm}p{1cm}p{6cm}} 
\begin{equation} \eqlabel{mod coalg-r}
\gbeg{4}{5}
\got{2}{\F(gh)} \got{2}{\F(id)} \gnl
\gcmu \gcmu \gnl%
\gcl{1} \glmptb \gnot{\hspace{-0,34cm}\nu_{h,1}} \grmptb \gcl{1} \gnl
\grm \grm \gnl
\gob{1}{\F(g)} \gob{3}{\F(h)} \gnl
\gend
=
\gbeg{3}{5}
\gvac{1} \got{1}{\F(gh)} \got{3}{\F(id)} \gnl
\gvac{1} \gcl{1} \gcn{1}{1}{3}{1} \gnl
\gvac{1} \grm \gnl
\gwcm{3} \gnl
\gob{1}{\F(g)} \gob{3}{\F(h)} \gnl
\gend,
\end{equation} & &
\begin{equation}\eqlabel{comod alg-r}
\gbeg{4}{5}
\got{2}{\F(g)} \got{2}{\F(h)} \gnl
\grcm \grcm \gnl
\gcl{1} \glmptb \gnot{\hspace{-0,34cm}\nu_{1,h}} \grmptb \gcl{1} \gnl
\gmu \gmu \gnl%
\gob{1}{\F(gh)} \gob{4}{\F(id)} \gnl
\gend
=
\gbeg{4}{5}
\got{1}{\F(g)} \gvac{1} \got{1}{\F(h)} \gnl
\gwmu{3} \gnl
\gvac{1} \grcm \gnl
\gvac{1} \gcl{1} \gcn{1}{1}{1}{3} \gnl
\gvac{1} \gob{1}{\F(gh)} \gob{3}{\F(id)} \gnl
\gend
\end{equation}
\end{tabular}
\end{center}
for 1-cells $A\stackrel{k}{\to}B\stackrel{h}{\to}B\stackrel{f}{\to}B \stackrel{g}{\to}C$.

\section{2-category of bilax functors} \selabel{2-cat Bilax} 

In this section we introduce the rest of the ingredients to construct a 2-category of bilax functors.

\subsection{Bilax natural transformations}  \sslabel{b.nat-tr}

Among bilax functors we introduce bilax natural transformations. Recall that
for a lax transformation $\psi$ and a colax transformation $\phi$, both acting between (bilax) functors $\F\Rightarrow\F'$,
for every 1-cell $f: A\to B$ there are 2-cells
$$\psi_f: \F'_{A,B}(f)\comp\psi(A) \Rightarrow  \psi(B)\comp\F_{A,B}(f) $$
and
$$\phi_f: \phi(B)\comp\F_{A,B}(f) \Rightarrow  \F'_{A,B}(f)\comp\phi(A) $$
natural in $f$. For the sake of the following definition we introduce the notation:
\begin{equation} \eqlabel{lambda}
\lambda_{xy,z}:=
\gbeg{4}{5}
\got{2}{\F(xy)} \got{2}{\F(z)} \gnl
\gcmu \gcl{1} \gnl
\gcl{1} \glmptb \gnot{\hspace{-0,34cm}\nu_{y,z}} \grmptb \gnl
\gmu \gcl{1}  \gnl
\gob{2}{\F(xz)} \gob{2}{\F(y)} \gnl
\gend
\end{equation}
for a bilax functor $\F$ and 1-cells  $A \stackrel{z}{\to} A \stackrel{y}{\to} A \stackrel{x}{\to} B$. Then the first equation in
\equref{bilax} can be expressed also as:
\begin{equation} \eqlabel{bilax-lambda} 
\gbeg{3}{5}
\got{0}{\F(gf)} \got{4}{\F(hk)} \gnl
\gcn{1}{1}{0}{1} \gcmu \gnl
\glmptb \gnot{\hspace{-0,34cm}\lambda_{gf,h}} \grmptb \gcl{1} \gnl
\gcn{1}{1}{1}{0} \gmu \gnl
\gob{0}{\F(gh)} \gob{4}{\F(fk)} \gnl
\gend
=
\gbeg{5}{5}
\got{1}{\F(gf)} \got{3}{\F(hk)} \gnl
\gwmu{3} \gnl
\glmp \gcmptb \gnot{\hspace{-0,78cm}\F(1c1)} \grmp \gnl
\gwcm{3} \gnl
\gob{1}{\F(gh)}\gvac{1}\gob{1}{\F(fk)}
\gend
\end{equation}
Observe that by \equref{eps-e} and the rules of the (co)lax structures of $\F$, one has:
\begin{equation} \eqlabel{lambda-x1}
\gbeg{2}{5}
\got{2}{\F(x)}  \gnl
\gcl{1} \gu{1} \gnl
\glmptb \gnot{\hspace{-0,38cm}\lambda_{x\hspace{0,04cm} id,id}} \grmptb \gnl
\gcl{1} \gcu{1} \gnl
\gob{1}{\F(x)} \gnl
\gend=\Id_{\F(x)}.
\end{equation}

\begin{defn}
A bilax natural transformation $\chi: \F\Rightarrow\F'$ between bilax functors is a pair $(\psi,\phi)$ consisting of a lax
natural transformation $\psi$ of lax functors and a colax natural transformation $\phi$ of colax functors,
which agree on the 1-cell components, {\em i.e.}
$\psi(A)=\phi(A):=\chi(A)$ for every $A\in\Ob\K$, and whose 2-cell components are related through the relation:
\begin{equation}\eqlabel{psi-lambda-phi for bimonads}
\gbeg{4}{7}
\got{-1}{\F'(xy)} \gvac{1} \got{3}{\chi(A)} \got{1}{\F(z)} \gnl
\gcn{1}{1}{0}{1} \gcl{1} \gcn{2}{2}{3}{1} \gnl
\glmptb \gnot{\hspace{-0,34cm}\psi_{xy}} \grmptb \gnl
\gcl{1} \glmptb \gnot{\hspace{-0,34cm}\lambda_{xy,z}} \grmptb \gnl
\glmptb \gnot{\hspace{-0,34cm}\phi_{xz}} \grmptb \gcn{2}{2}{1}{3} \gnl
\gcn{1}{1}{1}{0} \gcl{1} \gnl
\gob{-1}{\F'(xz)} \gvac{1} \gob{3}{\chi(A)} \gob{1}{\F(y)}
\gend=
\gbeg{3}{7}
\got{1}{\F'(xy)} \got{3}{\chi(A)} \got{1}{\F(z)}  \gnl
\gcn{2}{2}{0}{3} \gcl{1} \gcn{1}{1}{3}{1} \gnl
\gvac{1} \gvac{1} \glmptb \gnot{\hspace{-0,34cm}\phi_{z}} \grmptb \gnl
\gvac{1} \glmptb \gnot{\hspace{-0,34cm}\lambda'_{xy,z}} \grmptb \gcl{1} \gnl
\gcn{2}{2}{3}{0}  \glmptb \gnot{\hspace{-0,34cm}\psi_y} \grmptb \gnl
\gvac{2} \gcl{1} \gcn{1}{1}{1}{3} \gnl
\gob{1}{\F'(xz)} \gvac{1} \gob{1}{\chi(A)} \gob{3}{\F(y)}
\gend
\end{equation}
for composable 1-cells: $A\stackrel{z}{\to} A \stackrel{y}{\to} A \stackrel{x}{\to} B$.
We will denote it shortly as a triple $(\chi, \psi, \phi)$.
\end{defn}

In particular, if $y=z=id_A$ and one applies the unity of the lax structure of $\F$ on the top right in \equref{psi-lambda-phi for bimonads},
and the counity of the colax structure of $\F$ on the bottom right, one obtains (by \equref{lambda-x1} and \prref{(co)mod str}):
\begin{equation}\eqlabel{YD}
\gbeg{2}{6}
\got{-1}{\F'(x)} \gvac{1} \got{3}{\chi(A)} \gnl
\gcn{1}{1}{0}{1} \gcl{1} \gnl
\glmptb \gnot{\hspace{-0,34cm}\psi_{x}} \grmptb \gnl
\glmptb \gnot{\hspace{-0,34cm}\phi_{x}} \grmptb  \gnl
\gcn{1}{1}{1}{0} \gcl{1} \gnl
\gob{-1}{\F'(x)} \gvac{1} \gob{3}{\chi(A)}
\gend=
\gbeg{3}{5}
\got{1}{\F'(x)} \got{3}{\chi(A)}  \gnl
\gcl{1}  \glcm \gnl
\glmptb \gnot{\hspace{-0,34cm}\lambda'_{x\hspace{0,04cm} id,id}} \grmptb \gcl{1} \gnl
\gcl{1}\glm \gnl
\gob{1}{\F'(x)} \gob{3}{\chi(A)}
\gend
=
\gbeg{4}{5}
\got{2}{\F'(x)} \got{3}{\chi(A)} \gnl
\gcmu \glcm \gnl
\gcl{1} \glmptb \gnot{\hspace{-0,34cm}\nu_{id,id}} \grmptb \gcl{1} \gnl
\gmu \glm \gnl
\gob{2}{\F'(x)} \gob{3}{\chi(A)} \gnl
\gend
\end{equation}
which we will call the {\em Yetter-Drinfel`d condition on the bilax natural transformation $\chi$}.
Observe that the left module and comodule structures of $\chi(A)$ above are over $\F'(id_A)$.
The relation \equref{psi-lambda-phi for bimonads} we will call {\em strong Yetter-Drinfel`d condition on $\chi$},
and the 1-cells $\chi(A)$ in $\K'$ we will call {\em strong Yetter-Drinfel`d modules} (over $\F'(id_A)$).

\medskip

\begin{ex} \exlabel{bimnd-K}
Let $\K$ and $\K'$ be 2-categories with Yang-Baxter operators $c$ and $d$, respectively.
We fix two $d$-bimonads $b_0:A_0\to A_0$ and $b_1:A_1\to A_1$ in $\K'$ and define 
bilax functors $\F_0,\F_1:\K\to\K'$ by $\F_0:=\F_{b_0}$ and $\F_1:=\F_{b_1}$ with $\nu_i=d$ for $i=1,2$, as in \exref{my bimonads}.

A bilax natural transformation $(\chi, \psi, \phi)$ between bilax functors $\chi: \F_0\Rightarrow\F_1$
consists of a family of 2-cells
$$\psi_A: b_1\chi(A)\Rightarrow\chi(A) b_0 \quad \text{and}\quad \phi_A: \chi(A) b_0 \to b_1 \chi(A),$$
indexed by $A\in\Ob\K$, where $\psi_A$ (respectively $\phi_A$) are distributive laws with respect to the monad (resp. comonad)
structures of $b_0,b_1$ (by \equref{phi-colax} and its vertical dual), and 
it holds:
\begin{equation}  \eqlabel{YD-cond}
\gbeg{3}{5}
\got{1}{b_1} \got{1}{\chi_A} \got{1}{b_0} \gnl
\glmptb \gnot{\hspace{-0,34cm}\psi_A} \grmptb \gcl{1} \gnl
\gcl{1} \glmptb \gnot{\hspace{-0,34cm}\lambda_0} \grmptb \gnl
\glmptb \gnot{\hspace{-0,34cm}\phi_A} \grmptb \gcl{1} \gnl
\gob{1}{b_1} \gob{1}{\chi_A} \gob{1}{b_0}
\gend
=
\gbeg{3}{5}
\got{1}{b_1} \got{1}{\chi_A} \got{1}{b_0} \gnl
\gcl{1} \glmptb \gnot{\hspace{-0,34cm}\phi_A} \grmptb \gnl
\glmptb \gnot{\hspace{-0,34cm}\lambda_1} \grmptb \gcl{1} \gnl
\gcl{1} \glmptb \gnot{\hspace{-0,34cm}\psi_A} \grmptb \gnl
\gob{1}{b_1} \gob{1}{\chi_A} \gob{1}{b_0,}
\gend
\quad\text{and consequently}\quad
\gbeg{2}{4}
\got{1}{b_1} \got{1}{\chi_A} \gnl
\glmptb \gnot{\hspace{-0,34cm}\psi_A} \grmptb \gnl
\glmptb \gnot{\hspace{-0,34cm}\phi_A} \grmptb  \gnl
\gob{1}{b_1} \gob{1}{\chi_A}
\gend=
\gbeg{4}{5}
\got{2}{b_1} \got{3}{\chi_A} \gnl
\gcmu \glcm \gnl
\gcl{1} \glmptb \gnot{\hspace{-0,34cm}d} \grmptb \gcl{1} \gnl
\gmu \glm \gnl
\gob{2}{b_1} \gob{3}{\chi_A.}
\gend
\end{equation}
Note that $\lambda_0$ and $\lambda_1$ have the form:
$$\lambda_0=
\gbeg{4}{5}
\got{2}{b_0} \got{1}{b_0} \gnl
\gcmu \gcl{1} \gnl
\gcl{1} \glmptb \gnot{\hspace{-0,34cm}d} \grmptb \gnl  
\gmu \gcl{1}  \gnl
\gob{2}{b_0} \gob{1}{b_0} \gnl
\gend
\quad\text{and}\qquad
\lambda_1=
\gbeg{3}{5}
\got{2}{b_1} \got{1}{b_1} \gnl
\gcmu \gcl{1} \gnl
\gcl{1} \glmptb \gnot{\hspace{-0,34cm}d} \grmptb \gnl  
\gmu \gcl{1}  \gnl
\gob{2}{b_1} \gob{1}{b_1} \gnl
\gend
$$
and that the third equation in \equref{phi-colax} is now trivial. 
For every $A\in\Ob\K$, the triple $(\chi(A), \psi_A, \phi_A)$ is a 1-cell in the 2-category $\Bimnd(\K')$ from \cite[Section 7]{F2}, 
which we mentioned at the beginning of \ssref{mod}. The 0-cells of $\Bimnd(\K)$ are bimonads in $\K$ defined via a distributive 
law 2-cell $\lambda$, 1-cells are triples $(F,\psi,\phi)$ where $(F,\psi)$ is a 1-cell in $\Mnd(\K)$ and $(F,\phi)$ is a 1-cell 
in $\Comnd(\K)$ with a compatibility condition between $\psi, \phi$ and $\lambda$ (as in \equref{YD-cond} on the left), 
and a 2-cell is a single 2-cell $\zeta$ in $\K$ which is simultaneously a 2-cell in $\Mnd(\K)$ and in $\Comnd(\K)$. 
\end{ex}

\begin{ex} \exlabel{L-trivial: bimnd} 
If $\K=Del(\C)$ in the above Example is induced by a braided monoidal category $\C$,
the above bilax natural transformation $\chi: \F_0\Rightarrow\F_1:Del(\C)\to\K'$ is precisely a single 1-cell in $\Bimnd(Del(\C))$.
\end{ex}

\begin{ex} \exlabel{1-cell for K=1}
By \leref{T(id)} actually {\em any} two bilax functors $(\Tau_0,\nu),(\Tau_1,\nu'):1\to\K$ determine two bimonads in $\K$:  
$\T_0$ yields a $\nu$-bimonad $b_0:=\T(id_*)$ on $A=\T(*)$ and $\T_1$ a $\nu'$-bimonad $b_1:=\T_1(id_*)$ 
on $A'=\T'(*)$. Then analogously as in \exref{L-trivial: bimnd}, any bilax natural transformation $\chi: \Tau_0\Rightarrow\Tau_1:1\to\K'$ is 
precisely a single 1-cell in $\Bimnd(\K)$.
\end{ex}

\begin{ex} \exlabel{bilax trans triv} 
Consider a bilax transformation $\chi:(\T,\nu)\Rightarrow(\T',\nu'): 1\to\K$ from the trivial 2-category to a 2-category $\K$. 
Let $B:=\T(id_*)$ be the $\nu$-bimonad on $\A=\T(*)$ and $B':=\T'(id_*)$ the $\nu'$-bimonad  
on $\T'(*)$ as in \leref{T(id)}. Let $m(B)$ denote the monad part of $B$ and $c(B)$ the comonad part of $B$, and set $\chi(*)=X$. We find that 
$\psi:m(B')X\Rightarrow Xm(B)$ is a distributive law with respect to monads, and 
that $\phi:Xc(B)\Rightarrow c(B')X$ is a distributive law with respect to comonads. 

It is a nice exercise to prove the following lemma that we will use to pursue with this example. 

\begin{lma} \lelabel{ni-lambda}
Let $B$ be both a monad (to which we refer to as $m(B)$) and a comonad (we refer to it as $c(B)$) and suppose 
that $\nu:BB\Rightarrow BB$ is a distributive law, both on left and right side, with respect to monad $m(B)$ 
and with respect to comonad $c(B)$ (this means four distributive laws). Then 
$$\lambda:=
\gbeg{4}{5}
\got{2}{m(B)} \got{2}{c(B)} \gnl
\gcmu \gcl{1} \gnl
\gcl{1} \glmptb \gnot{\hspace{-0,34cm}\nu} \grmptb \gnl
\gmu \gcl{1}  \gnl
\gob{2}{c(B)} \gob{2}{m(B)} \gnl
\gend$$
is a distributive law on the left both with respect to monads and comonads, that is: 
\begin{equation}\eqlabel{lambda d.l.}
\gbeg{3}{5}
\got{0}{m(B)} \got{3}{m(B)} \got{0}{c(B)} \gnl
\gmu \gcn{1}{1}{2}{0} \gnl
\gvac{1} \hspace{-0,34cm} \glmptb \gnot{\hspace{-0,34cm}\lambda} \grmptb  \gnl
\gvac{1} \gcl{1} \gcl{1} \gnl
\gvac{1} \gob{0}{c(B)} \gob{3}{m(B)}
\gend
=
\gbeg{4}{5}
\gvac{1} \got{0}{m(B)} \gvac{1} \got{1}{m(B)} \got{2}{c(B)} \gnl
\gvac{1} \gcn{1}{1}{0}{1} \glmpt \gnot{\hspace{-0,34cm}\lambda} \grmptb \gnl
\gvac{1} \glmptb \gnot{\hspace{-0,34cm}\lambda} \grmptb \gcl{1} \gnl
\gvac{1} \gcl{1} \gmu \gnl
\gvac{1} \gob{1}{c(B)} \gob{2}{m(B)}
\gend, 
\quad
\gbeg{2}{5}
\got{1}{}\got{2}{c(B)} \gnl
\gu{1} \gcl{1} \gnl
\glmptb \gnot{\hspace{-0,34cm}\lambda} \grmptb \gnl
\gcl{1} \gcl{1} \gnl
\gob{0}{c(B)} \gob{3}{m(B)}
\gend=
\gbeg{2}{5}
\got{1}{c(B)} \gnl
\gcl{1} \gu{1} \gnl
\gcl{2}  \gcl{2} \gnl
\gob{0}{c(B)} \gob{3}{m(B)}
\gend;
\qquad
\gbeg{3}{5}
\got{1}{m(B)}\got{2}{c(B)} \gnl
\gcl{1} \gcmu \gnl
\glmptb \gnot{\hspace{-0,34cm}\lambda} \grmptb \gcl{1} \gnl
\gcn{1}{1}{1}{-1} \glmptb \gnot{\hspace{-0,34cm}\lambda} \grmptb \gnl
\gob{-1}{c(B)} \gvac{2} \gob{0}{c(B)} \gob{3}{m(B)}
\gend=
\gbeg{3}{5}
\got{1}{m(B)}\got{2}{c(B)} \gnl
\gcn{1}{1}{2}{2} \gcn{1}{1}{2}{2} \gnl
\gvac{1} \hspace{-0,34cm} \glmpt \gnot{\hspace{-0,34cm}\lambda} \grmptb \gnl
\gvac{1} \hspace{-0,2cm} \gcmu \gcn{1}{1}{0}{1} \gnl
\gvac{1} \gob{0}{c(B)} \gvac{1} \gob{1}{c(B)} \gob{2}{m(B)}
\gend,
\quad
\gbeg{2}{5}
\got{1}{m(B)}\got{2}{c(B)} \gnl
\gcl{1} \gcl{1} \gnl
\glmptb \gnot{\hspace{-0,34cm}\lambda} \grmptb \gnl
\gcu{1} \gcl{1} \gnl
\gob{3}{m(B)}
\gend=
\gbeg{3}{5}
\got{1}{m(B)}\got{2}{c(B)} \gnl
\gcl{1} \gcl{1} \gnl
\gcl{2}  \gcu{1} \gnl
\gob{1}{m(B).}
\gend
\end{equation}
\end{lma}

\medskip

Continuing with the Example, we have that similarly $\nu'$ induces $\lambda'$. Moreover, the compatibilities 
\begin{equation}\eqlabel{psi-lambda-phi J. Power}
\gbeg{3}{5}
\got{1}{m(B')} \got{1}{X} \got{1}{c(B)} \gnl
\glmptb \gnot{\hspace{-0,34cm}\psi} \grmptb \gcl{1} \gnl
\gcl{1} \glmptb \gnot{\hspace{-0,34cm}\lambda} \grmptb \gnl
\glmptb \gnot{\hspace{-0,34cm}\phi} \grmptb \gcl{1} \gnl
\gob{1}{c(B')} \gob{1}{X} \gob{1}{m(B)}
\gend=
\gbeg{3}{5}
\got{1}{m(B')} \got{1}{X} \got{1}{c(B)}  \gnl
\gcl{1} \glmptb \gnot{\hspace{-0,34cm}\phi} \grmptb \gnl
\glmptb \gnot{\hspace{-0,34cm}\lambda'} \grmptb \gcl{1} \gnl
\gcl{1}  \glmptb \gnot{\hspace{-0,34cm}\psi} \grmptb \gnl
\gob{1}{c(B')} \gob{1}{X} \gob{1}{m(B)}
\gend
\qquad\text{and}\qquad
\gbeg{2}{6}
\got{-1}{m(B')} \gvac{1} \got{3}{X} \gnl
\gcn{1}{1}{0}{1} \gcl{1} \gnl
\glmptb \gnot{\hspace{-0,34cm}\psi} \grmptb \gnl
\glmptb \gnot{\hspace{-0,34cm}\phi} \grmptb  \gnl
\gcn{1}{1}{1}{0} \gcl{1} \gnl
\gob{-1}{c(B')} \gvac{1} \gob{3}{X}
\gend=
\gbeg{3}{5}
\got{1}{m(B)} \got{3}{X}  \gnl
\gcl{1}  \glcm \gnl
\glmptb \gnot{\hspace{-0,34cm}\lambda'} \grmptb \gcl{1} \gnl
\gcl{1}\glm \gnl
\gob{1}{c(B')} \gob{3}{X}
\gend
\end{equation}
hold. 
Then $(X,\psi,\phi):(\A,m(B), c(B),\lambda)\to(\A',m(B'), c(B'),\lambda')$ is a 1-cell in the 2-category of mixed distributive laws 
$\Dist(\K)$ of \cite[Definition 6.2]{PW}. (In the specific case when $X$ is a left $m(B')$-module and left $c(B')$-comodule, and 
$\psi$ and $\phi$ are given by 
$\psi=
\gbeg{3}{3}
\got{1}{m(B')} \got{1}{X} \gnl
\glm\ \gu{1}  \gnl
\gvac{1} \gob{1}{X} \gob{1}{m(B)}
\gend$ 
and 
$\phi=
\gbeg{3}{3}
\gvac{1} \got{1}{X} \got{1}{c(B)} \gnl
\glcm \gcu{1} \gnl
\gob{1}{c(B')} \gob{1}{X} 
\gend$, 
the two expressions in \equref{psi-lambda-phi J. Power} are equivalent and 
one recovers a particular form of $\lambda$-bialgebras by Turin and Plotkin, \cite[Section 7.2]{TP}.)  
\end{ex}

Suppose that $B$ is a bialgebra in a braided monoidal category $\C$. A (left) Yetter-Drinfel`d module over $B$ is an object $M$ together 
with a (left) action $B\ot M\to M$ and a (left) coaction $M\to B\ot M$ of $B$ subject to the compatibility condition: 
\begin{equation} \eqlabel{YD-def}
\gbeg{3}{8}
\got{2}{B} \got{1}{M} \gnl
\gcmu \gcl{1} \gnl
\gcl{1} \gbr \gnl
\glm \gcl{2} \gnl
\glcm \gnl
\gcl{1} \gbr \gnl
\gmu \gcl{1} \gnl
\gob{2}{B} \gob{1}{M} 
\gend
=
\gbeg{4}{5}
\got{2}{B} \got{3}{M} \gnl
\gcmu \glcm \gnl
\gcl{1} \gbr \gcl{1} \gnl
\gmu \glm \gnl
\gob{2}{B} \gob{3}{M.}
\gend
\end{equation}
The category of (left) Yetter-Drinfel`d modules over $B$ in $\C$ and left $B$-linear and $B$-colinear morphisms we denote by 
${}^B_B\YD(\C)$.

\begin{rem} \rmlabel{bialg vs Hopf}
Observe that the antipode, {\em i.e.} a Hopf algebra structure on a bialgebra in the context of Yetter-Drinfel'd modules, is
used in the following two instances. One is to construct the inverse for the braiding of the respective category. Another one is 
to formulate an equivalent condition to \equref{YD-def}. Thus, the category of Yetter-Drinfel'd modules
over a bialgebra is monoidal and even it has a pre-braiding (non-invertible), given by: 
$$
\gbeg{3}{5}
\gvac{1} \got{1}{M} \got{1}{N} \gnl
\glcm \gcl{1} \gnl
\gcl{1} \gbr \gnl
\glm \gcl{1} \gnl
\gvac{1} \gob{1}{N} \gob{1}{M.}
\gend
$$
\end{rem}

\begin{ex} \exlabel{triv-braided}
Consider two braided monoidal categories $\C$ and $\D$ and two bialgebras $B_0, B_1$ in $\D$. These give rise to two bilax 
(and bimonoidal) functors $F_{B_0}, F_{B_1}: Del(\C)\to Del(\D)$ as in \exref{my bimonads}. 
The bilax natural transformation $\chi: \F_0\Rightarrow\F_1:\K\to\K'$ from \exref{bimnd-K} corresponds to a generalized notion of a
Yetter-Drinfel`d module over $B_1$, which in view of the above we call {\em strong Yetter-Drinfel`d modules}.
\end{ex}

\begin{prop} \prlabel{YD specific}
Any Yetter-Drinfel'd module $M$ over a bialgebra $B'$ in a braided monoidal category $(\C,\Phi)$ comes from a bilax natural transformation
of \exref{triv-braided} where $\psi$ and $\phi$ are given by
\begin{equation} \eqlabel{left-left YD classic}
\psi=
\gbeg{3}{5}
\got{2}{B'} \got{1}{M} \gnl
\gcmu \gcl{1} \gnl
\gcl{1} \gbr \gnl
\glm \gbmp{\s j^{-1}} \gnl
\gvac{1} \gob{1}{M} \gob{1}{B}
\gend\hspace{0,16cm}, \quad
\phi=
\gbeg{3}{5}
\gvac{1} \got{1}{M} \got{1}{B} \gnl
\glcm \gbmp{j} \gnl
\gcl{1} \gbr \gnl
\gmu \gcl{1} \gnl
\gob{2}{B'} \gob{1}{M}
\gend\hspace{0,16cm}, \quad
\end{equation}
for any bialgebra isomorphism $j:B\to B'$. (More precisely, 
from a bilax endo-transformation 
with $j=id_B$ in \equref{left-left YD classic}.) 
\end{prop}

\begin{proof}
The notation in these two diagrams is the usual one for braided monoidal categories, concretely
$\gbeg{2}{1}
\gmu \gnl
\gend$ and
$\gbeg{2}{1}
\gcmu \gnl
\gend$ stand for the (co)multiplication of $B'$. 
That the given $\psi$ and $\phi$ are desired distributive laws ({\em i.e.} (co)lax natural transformations) it was proved at 
the beginning of \cite[Section 5.1]{F1}, though for $B=B'$ and trivial $j$.
The algebra (resp. coalgebra) morphism property of $j^{-1}$ (resp. $j$) makes \equref{left-left YD classic} the desired distributive laws
for nontrivial $j$. The first claim now follows from \cite[Proposition 7.5]{F2}, whose conditions are fulfilled since $\C$ is braided.
Set $\nu_{F,X}=(M\ot j^{-1})\Phi_{F\s',X}$ and $\nu_{X,F}=\Phi_{X,F\s'}(M\ot j)$. The second claim follows from Corollary 7.6 of 
{\em loc. cit.}. 
\qed\end{proof}

\begin{ex} \exlabel{both-braided}
Let $(\chi, \psi, \phi)$ be a bilax natural transformation between bilax functors with compatible Yang Baxter operators 
$\chi: \F\Rightarrow\F':Del(\C)\to Del(\D)$
where $\C$ and $\D$ are braided monoidal categories 
with braidings $\Phi_\C$ and $\Phi_\D$, respectively. 
Then $F:=\F_{*,*}$ and $G:=\F'_{*,*}$ are bimonoidal functors $\C\to\D$ as in \exref{braided}, $\chi(*)=M$ is an object in $\D$
and there are morphisms
$$\psi_X: G(X)\ot M \to  M\ot F(X) \quad \text{and}\quad \phi_X: M\ot F(X) \to  G(X)\ot M$$
natural in $X\in\C$, where $\psi$ is a distributive law for the monoidal functor structures, and $\phi$ is a distributive law
for the comonoidal functor structures, so that the left identity below holds, and consequently the one next to it:
$$
\gbeg{4}{7}
\got{-1}{G(XY)} \gvac{1} \got{3}{M} \got{1}{F(Z)} \gnl
\gcn{1}{1}{0}{1} \gcl{1} \gcn{2}{2}{3}{1} \gnl
\glmptb \gnot{\hspace{-0,34cm}\psi_{XY}} \grmptb \gnl
\gcl{1} \glmptb \gnot{\hspace{-0,34cm}\lambda_{XY,Z}} \grmptb \gnl
\glmptb \gnot{\hspace{-0,34cm}\phi_{XZ}} \grmptb \gcn{2}{2}{1}{3} \gnl
\gcn{1}{1}{1}{0} \gcl{1} \gnl
\gob{-1}{G(XZ)} \gvac{1} \gob{3}{M} \gob{1}{F(Y)}
\gend=
\gbeg{6}{7}
\got{1}{G(XY)} \got{3}{M} \got{1}{F(Z)}  \gnl
\gcn{2}{2}{0}{3} \gcl{1} \gcn{1}{1}{3}{1} \gnl
\gvac{1} \gvac{1} \glmptb \gnot{\hspace{-0,34cm}\phi_{Z}} \grmptb \gnl
\gvac{1} \glmptb \gnot{\hspace{-0,34cm}\lambda'_{XY,Z}} \grmptb \gcl{1} \gnl
\gcn{2}{2}{3}{0}  \glmptb \gnot{\hspace{-0,34cm}\psi_Y} \grmptb \gnl
\gvac{2} \gcl{1} \gcn{1}{1}{1}{3} \gnl
\gob{1}{G(XZ)} \gvac{1} \gob{1}{M} \gob{3}{F(Y),}
\gend
\qquad
\gbeg{2}{4}
\got{1}{G(X)} \got{1}{M} \gnl
\glmptb \gnot{\hspace{-0,34cm}\psi_X} \grmptb \gnl
\glmptb \gnot{\hspace{-0,34cm}\phi_X} \grmptb  \gnl
\gob{1}{G(X)} \gob{1}{M}
\gend=
\gbeg{4}{5}
\got{2}{G(X)} \got{3}{M} \gnl
\gcmu \glcm \gnl
\gcl{1} \gbr \gcl{1} \gnl
\gmu \glm \gnl
\gob{2}{G(X)} \gob{3}{M,}
\gend
\quad\text{with}\quad
\lambda_{XY,Z}:=
\gbeg{4}{5}
\got{2}{F(XY)} \got{2}{F(Z)} \gnl
\gcmu \gcl{1} \gnl
\gcl{1} \gbr \gnl
\gmu \gcl{1}  \gnl
\gob{2}{F(XZ)} \gob{2}{F(Y).} \gnl
\gend
$$
For bialgebras $B$ in $\C$ by \prref{preserve bimnd} $F(B), G(B)$ are bialgebras in $\D$, 
and if $\psi_B, \phi_B$ are of the form as in \equref{left-left YD classic}, we recover classical Yetter-Drinfel`d modules in $\C$.
\end{ex}

The bilax natural transformations, {\em i.e.} identities \equref{psi-lambda-phi for bimonads} and \equref{YD}, offer the following
point of view. Given any monoidal category $\D$, when one considers the center category $\Z_l^w(\D)$, one is given a family of colax 
transformations $\phi$. In particular, when $\D={}_H\C$ the category of modules over a bialgebra or a Hopf algebra $H$ in a braided
monoidal category $\C$, one is able to construct lax transformations $\psi$ (as in \equref{left-left YD classic}) so that the given
$\phi$ and this $\psi$ obey \equref{YD} - since $\phi$ is $H$-linear, being a morphism in ${}_H\C$. Similarly, considering $\Z_r^w({}^H\C)$
one is given $\psi$'s and one constructs $\phi$'s, so that they together obey \equref{YD}. As in the proof of the above Proposition (that is,
as proved in \cite[Proposition 7.5]{F2}), in this setting the bilax condition \equref{psi-lambda-phi for bimonads} follows.

\subsection{Bilax modifications}

We finally introduce: 

\begin{defn}
Let $\chi, \chi' : (\F, \nu) \Rightarrow (\F', \nu') : (\K, c) \to \K'$ be bilax natural transformations. 
A {\em bilax modification} $a: \chi \Rrightarrow \chi'$ is a collection of 2-cells $(a(A))_{A\in\Ob(\K)}$ satisfying equations:
\begin{center} \hspace{-0,6cm}
\begin{tabular}{p{6.6cm}p{-1cm}p{8.6cm}}
\begin{equation}  \eqlabel{lax modif}
\gbeg{3}{6}
\got{1}{\chi_B}\got{1}{\F(x)} \gnl
\gcl{1} \gcl{1} \gnl
\glmptb \gnot{\hspace{-0,34cm}\phi_x} \grmptb \gnl
\gcl{1} \glmptb \gnot{\hspace{-0,34cm}a(A)} \grmp \gnl
\gcl{1} \gcl{1} \gnl
\gob{0}{\F'(x)} \gob{3}{\chi'_A}
\gend=
\gbeg{2}{6}
\gvac{1} \got{1}{\chi_B}\got{1}{\F(x)} \gnl
\gvac{1} \gcl{1} \gcl{1} \gnl
\glmp \gnot{\hspace{-0,34cm}a(B)} \grmptb\gcl{1} \gnl
\gvac{1} \glmptb \gnot{\hspace{-0,34cm}\phi'_x} \grmptb \gnl
\gvac{1} \gcl{1} \gcl{1} \gnl
\gvac{1} \gob{0}{\F'(x)} \gob{3}{\chi'_A}
\gend
\end{equation} & &
\begin{equation}  \eqlabel{colax modif}
\gbeg{3}{6}
\got{2}{\F'(x)} \got{1}{\chi_A} \gnl
\gvac{1} \gcl{1} \gcl{1} \gnl
\gvac{1} \glmptb \gnot{\hspace{-0,34cm}\psi_x} \grmptb \gnl
\glmp \gnot{\hspace{-0,34cm}a(B)} \grmptb\gcl{1} \gnl
\gvac{1} \gcl{1} \gcl{1} \gnl
\gvac{1}  \gob{0}{\chi'_B} \gob{3}{\F(x)}
\gend
=
\gbeg{3}{6}
\got{0}{\F'(x)} \got{3}{\chi_A} \gnl
\gcl{1} \gcl{1} \gnl
\gcl{1} \glmptb \gnot{\hspace{-0,34cm}a(A)} \grmp \gnl
\glmptb \gnot{\hspace{-0,34cm}\psi'_x} \grmptb \gnl
\gcl{1} \gcl{1} \gnl
\gob{1}{\chi'_B} \gob{2}{\F(x).}
\gend
\end{equation}
\end{tabular}
\end{center}
\end{defn}

Equivalently, a bilax modification is a modification both of lax and colax natural transformations: 
$a: \psi\Rrightarrow\psi'$ and $a: \phi\Rrightarrow\phi'$, 
where $(\psi,\phi)$ constitute $\chi$ and $(\psi',\phi')$ constitute $\chi'$.

\begin{ex} \exlabel{2-cells Bimnd}
Pursuing \exref{1-cell for K=1} a bilax modification between bilax natural transformations of bilax functors $1\to\K$ is precisely a 2-cell in $\Bimnd(\K)$. 
\end{ex}

\begin{ex} \exlabel{bilax modif} 
Recall \exref{bilax trans triv} where bilax natural transformations are 1-cells in the 2-category of mixed distributive laws 
$\Dist(\K)$ of \cite[Definition 6.2]{PW}. In this setting a bilax modification of bilax natural transformations is a 2-cell 
$\zeta: X\Rightarrow Y$ in $\K$ that satisfies:
$$
\gbeg{3}{4}
\got{1}{m(B')} \got{1}{X} \gnl
\glmptb \gnot{\hspace{-0,34cm}\psi} \grmptb \gnl
\gbmp{\zeta} \gcl{1} \gnl
\gob{1}{Y} \gob{1}{m(B)}
\gend
=
\gbeg{3}{4}
\got{1}{m(B')} \got{1}{X} \gnl
\gcl{1} \gbmp{\zeta} \gnl
\glmptb \gnot{\hspace{-0,34cm}\psi'} \grmptb \gnl
\gob{1}{Y} \gob{1}{m(B)}
\gend
\qquad \text{and} \qquad
\gbeg{2}{4}
\got{1}{X}\got{1}{c(B)} \gnl
\glmptb \gnot{\hspace{-0,34cm}\phi} \grmptb \gnl
\gcl{1} \gbmp{\zeta} \gnl
\gob{1}{c(B')} \gob{1}{Y}
\gend=
\gbeg{2}{4}
\got{1}{X}\got{1}{c(B)} \gnl
\gbmp{\zeta} \gcl{1} \gnl
\glmptb \gnot{\hspace{-0,34cm}\phi'} \grmptb \gnl
\gob{1}{c(B')} \gob{1}{Y.}
\gend
$$
As such it is a 2-cell in the 2-category $\Dist(\K)$ of \cite[Definition 6.2]{PW}. 
\end{ex}

\begin{ex} 
In the setting of \exref{triv-braided}, where bilax natural transformations are strong Yetter-Drinfel`d modules, 
a bilax modification of bilax natural transformations is a morphism $f$ in $\D$ satisfying:
$$
\gbeg{3}{4}
\got{1}{B'} \got{1}{M} \gnl
\glmptb \gnot{\hspace{-0,34cm}\psi} \grmptb \gnl
\gbmp{f} \gcl{1} \gnl
\gob{1}{N} \gob{1}{B}
\gend
=
\gbeg{3}{4}
\got{1}{B'} \got{1}{M} \gnl
\gcl{1} \gbmp{f} \gnl
\glmptb \gnot{\hspace{-0,34cm}\psi'} \grmptb \gnl
\gob{1}{N} \gob{1}{B}
\gend
\qquad \text{and} \qquad
\gbeg{2}{4}
\got{1}{M}\got{1}{B} \gnl
\glmptb \gnot{\hspace{-0,34cm}\phi} \grmptb \gnl
\gcl{1} \gbmp{f} \gnl
\gob{1}{B'} \gob{1}{N}
\gend=
\gbeg{2}{4}
\got{1}{M}\got{1}{B} \gnl
\gbmp{f} \gcl{1} \gnl
\glmptb \gnot{\hspace{-0,34cm}\phi'} \grmptb \gnl
\gob{1}{B'} \gob{1}{N.}
\gend
$$
By \equref{left G(b)-comod} and \equref{left G(b)-mod} this means that $f$ is both a morphism of left $B'$-modules and
left $B'$-comodules. This is a morphism of strong Yetter-Drinfel'd modules from \exref{triv-braided}.
\end{ex}

Now we may formulate:

\begin{prop}
The category of Yetter-Drinfel`d modules ${}^B_B\YD(\D)$ over a bialgebra $B$ in a braided monoidal category $(\D,\Phi_\D)$ is a 
full subcategory of the category $\Bilax(\F_B)$ of bilax endo-transformations on a bilax functor $\F_B:(Del(\C),\Phi_\C)\to (Del(\D),\Phi_\D)$ with compatible Yang-Baxter operator as in \exref{triv-braided} and bilax modifications.
\end{prop}

Similarly one has: 

\begin{prop} \prlabel{YD iso}
The category ${}^B_B\YD(\D)$ is isomorphic to the category $\Bilax(\T_B)$ of bilax endo-transformations on a bilax functor 
$\T_B:1\to (Del(\D),\Phi_\D)$ with compatible Yang-Baxter operator as in \leref{blx from triv} and bilax modifications.
\end{prop}

\subsection{2-category of bilax functors}

We finish this section by concluding that bilax functors (with compatible Yang-Baxter operator) $\K\to\K'$, 
bilax natural transformations and bilax modifications form a 2-category $\Bilax(\K,\K')$ (respectively, $\Bilax_c(\K,\K')$). 
The composition of bilax transformations $(\chi, \psi, \phi):\F\Rightarrow\G$ and $(\chi', \psi',\phi'): \G\Rightarrow\HH$
is easily seen to be induced by the vertical compositions of the (co)lax transformations $\psi'\cdot\psi, \phi'\cdot\phi$, namely by:
 $$(\psi'\cdot\psi)_f=
\gbeg{5}{6}
\got{-1}{\HH(f)} \gvac{2} \got{1}{\psi'(A)} \got{3}{\psi(A)} \gnl
\gcn{1}{1}{-1}{1} \gcl{1} \gcn{2}{2}{3}{1} \gnl
\glmptb \gnot{\hspace{-0,34cm}\psi'_f} \grmptb \gnl
\gcn{1}{2}{1}{-1} \glmptb \gnot{\hspace{-0,34cm}\psi_f} \grmptb \gnl
\gvac{1} \gcl{1} \gcn{1}{1}{1}{3} \gnl
\gob{-2}{\psi'(B)} \gvac{3} \gob{1}{\psi(B)} \gvac{1} \gob{2}{\F(f)}
\gend\qquad\text{and}\qquad
(\phi'\cdot\phi)_f=
\gbeg{5}{6}
\got{-1}{\phi'(B)} \gvac{2} \got{1}{\phi(B)} \got{3}{\F(f)} \gnl
\gcn{1}{2}{-1}{1} \gcl{1} \gcn{1}{1}{3}{1} \gnl
\gvac{1} \glmptb \gnot{\hspace{-0,34cm}\phi_f} \grmptb \gnl
\glmptb \gnot{\hspace{-0,34cm}\phi'_f} \grmptb \gcn{1}{2}{1}{2} \gnl
\gcn{1}{1}{1}{-1} \gcl{1} \gnl
\gob{-1}{\HH(f)} \gvac{2} \gob{1}{\phi'(A)} \gvac{1} \gob{1}{\phi(A).}
\gend
$$
Bilax modifications compose both horizontally and vertically, in the obvious and natural way. 

\medskip

We comment for the record that although the lax and colax natural transformations compose horizontally by:
\begin{equation} \eqlabel{comp colax}
(\psi'\comp\psi)_f=
\gbeg{5}{7}
\got{-1}{\G'\G(f)} \gvac{2} \got{2}{\psi'\G(A)} \gvac{1} \got{1}{\F'\psi(A)} \gnl
\gcn{1}{1}{-1}{1} \gcl{1} \gcn{2}{2}{4}{3} \gnl
\glmptb \gnot{\hspace{-0,34cm}\psi'_{\G(f)}} \grmptb \gnl
\gcl{1} \gwmu{3} \gnl
\gcn{1}{2}{1}{-2} \glmp \gcmptb \gnot{\hspace{-0,8cm}\F'(\psi_f)} \grmp \gnl
\gvac{1} \gwcm{3} \gnl
\gob{-2}{\psi'\G(B)} \gvac{3} \gob{1}{\F'\psi(B)} \gvac{1} \gob{2}{\F'\F(f)}
\gend\qquad\text{and}\qquad
(\phi'\comp\phi)_f=
\gbeg{5}{7}
\got{-2}{\phi'\G(B)} \gvac{3} \got{1}{\F'\phi(B)} \gvac{1} \got{2}{\F'\F(f)} \gnl
\gcn{1}{2}{-2}{1} \gwmu{3} \gnl
\gvac{1} \glmp \gcmptb \gnot{\hspace{-0,8cm}\F'(\phi_f)} \grmp \gnl
\gcl{1} \gwcm{3} \gnl
\glmptb \gnot{\hspace{-0,34cm}\phi'_{\G(f)}} \grmptb \gcn{2}{2}{3}{4} \gnl
\gcn{1}{1}{1}{-1} \gcl{1} \gnl
\gob{-1}{\G'\G(f)} \gvac{2} \gob{2}{\phi'\G(A)} \gvac{1} \gob{1}{\F'\phi(A),}
\gend
\end{equation}
the horizontal composition of lax and colax natural transformations does not induce a bilax
transformation. Namely, in order for this to work, the (co)lax structures should be identities.

\bigskip

We finally compare the 2-category $\Bilax(\K,\K')$, 
more precisely its special case $\Bilax(1,\K)$, 
with two other existing 2-categories in the literature, namely $\Dist(\K)$ and $\Bimnd(\K)$ mentioned before.   
From \leref{T(id)}, \exref{1-cell for K=1} and \exref{2-cells Bimnd} we clearly have:

\begin{prop}
There is a 2-category isomoprhism
$$\Bilax(1, \K)\iso\Bimnd(\K).$$
\end{prop}

From \leref{T(id)}, \exref{bilax trans triv} and \exref{bilax modif}, 
it can be appreciated that on the level of 1- and 2-cells there is a 
faithful assignment $\Bimnd(\K)\hookrightarrow\Dist(\K).$ 
Since the 0-cells of $\Dist(\K)$ are given by tupples $(\A, T, D, \lambda)$, where $T$ is a monad 
and $D$ a comonad on a 0-cell $\A$ in $\K$, and $\lambda:TD\Rightarrow DT$ is a distributive law with respect to monad and comonad as in 
\equref{lambda d.l.},  we clearly have:

\begin{prop}
There is a faithful 2-functor  
$$\Bimnd(\K)\hookrightarrow\Dist(\K),$$ 
which is defined on 0-cells by $(\A, B, \nu)\mapsto(\A, m(B),c(B), \lambda)$, with $\lambda$ being 
$$
\lambda(\nu):=
\gbeg{4}{5}
\got{2}{B} \got{1}{B} \gnl
\gcmu \gcl{1} \gnl
\gcl{1} \glmptb \gnot{\hspace{-0,34cm}\nu} \grmptb \gnl
\gmu \gcl{1}  \gnl
\gob{2}{B} \gob{1}{B.} \gnl
\gend
$$
\end{prop}


Observe that \prref{YD iso} 
is a consequence of the above 2-category isomorphism $\Bilax(1, \K)\iso\Bimnd(\K)$. 

\bigskip
\bigskip

{\bf Acknowledgments.} 
The first author was supported by the Science Fund of the Republic of Serbia, Grant No. 7749891, Graphical Languages - GWORDS. 

\bigskip


\end{document}